\documentclass[graybox]{svmult}

\usepackage{amsthm}
\usepackage{mathptmx}       % selects Times Roman as basic font
\usepackage{helvet}         % selects Helvetica as sans-serif font
\usepackage{courier}        % selects Courier as typewriter font
\usepackage{type1cm}        % activate if the above 3 fonts are
\usepackage{makeidx}         % allows index generation
\usepackage{graphicx}        % standard LaTeX graphics tool
\usepackage{multicol}        % used for the two-column index
\usepackage[bottom]{footmisc}% places footnotes at page bottom

\usepackage{amsmath}
\usepackage{amssymb}
\usepackage{amscd}
\usepackage{indentfirst}

\def\1{{\bf 1}}

\makeindex             % used for the subject index
\begin{document}

\newtheorem{thm}{Theorem}
\newtheorem*{myconjecture}{Conjecture}

\title*{Towards a fractal cohomology: Spectra of Polya--Hilbert operators, regularized determinants and Riemann zeros}
\titlerunning{Title}
\author{Tim Cobler and Michel L. Lapidus \thanks{The work of Michel L. Lapidus was supported by the U.S. National Science Foundation (NSF) under the grant DMS-1107750.}}
\authorrunning{Tim Cobler \and Michel L. Lapidus}
\institute{Tim Cobler \at Department of Mathematics, Fullerton College\\ \email{tcobler@fullcoll.edu} \and Michel L. Lapidus \at Department of Mathematics, University of California, Riverside\\ \email{lapidus@math.ucr.edu} }

\maketitle

\abstract{}
Emil Artin defined a zeta function for algebraic curves over finite fields and made a conjecture about them analogous to the famous Riemann hypothesis. This and other conjectures about these zeta functions would come to be called the Weil conjectures, which were proved by Weil in the case of curves and eventually, by Deligne in the case of varieties over finite fields. Much work was done in the search for a proof of these conjectures, including the development in algebraic geometry of a Weil cohomology theory for these varieties, which uses the Frobenius operator on a finite field. The zeta function is then expressed as a determinant, allowing the properties of the function to relate to the properties of the operator. The search for a suitable cohomology theory and associated operator to prove the Riemann hypothesis has continued to this day. In this paper we study the properties of the derivative operator $D = \frac{d}{dz}$ on a particular family of weighted Bergman spaces of entire functions on $\mathbb{C}$. The operator $D$ can be naturally viewed as the 'infinitesimal shift of the complex plane' since it generates the group of translations of $\mathbb{C}$. Furthermore, this operator is meant to be the replacement for the Frobenius operator in the general case and is used to construct an operator associated to any given meromorphic function. With this construction, we show that for a wide class of meromorphic functions, the function can be recovered by using a regularized determinant involving the operator constructed from the meromorphic function. This is illustrated in some important special cases: rational functions, zeta functions of algebraic curves (or, more generally, varieties) over finite fields, the Riemann zeta function, and culminating in a quantized version of the Hadamard factorization theorem that applies to any entire function of finite order. This shows that all of the information about the given meromorphic function is encoded into the special operator we constructed. Our construction is motivated in part by work of Herichi and the second author on the infinitesimal shift of the real line (instead of the complex plane) and the associated spectral operator, as well as by earlier work and conjectures of Deninger on the role of cohomology in analytic number theory, and a conjectural 'fractal cohomology theory' envisioned in work of the second author and of Lapidus and van Frankenhuijsen on complex fractal dimensions. 
%%%%%%%%%%%%%%%%%%%%%%%%%%%%%%%%%%%%%%%%%%%%%%%%%%%%%%%%%%%%%%%%%%%%%%%%%%%%%%%%%%%%%%%%%%%
%%%%%%%%%%%%%%%%%%%%%%%%%%%%%%%%%%%%%%%%%%%%%%%%%%%%%%%%%%%%%%%%%%%%%%%%%%%%%%%%%%%%%%%%%%%
\tableofcontents

\vspace{5mm}

\textbf{Suggested Running Title: Regularized determinants and Riemann zeros}

\section{Introduction}

Riemann's famous paper, \cite{Rie}, opened up the use of complex analysis to study the prime numbers. This approach has yielded many great results in number theory, including, but certainly not limited to, the Prime Number Theorem. Riemann also made his well-known conjecture that stands to this very day. We will refer to this as (RH) in this paper. 

\begin{myconjecture}
\emph{(RH)} The only nontrivial zeros of $\zeta(s)$ occur when $s$ satisfies $\Re(s) = \frac{1}{2}$. 
\end{myconjecture}

For further reading about $\zeta(s)$, see \cite{Tit}, \cite{Pat}, and \cite{Edw}. However, despite (RH) remaining unsolved over 150 years after it was made, there is an analogue for zeta functions of algebraic varieties over finite fields that has been proven. The development of this theory also introduced new techniques to number theory. We will begin with a short history of this result. 

\subsection{The Weil Conjectures}

Using the Euler product representation of the Riemann zeta function in terms of the rational primes, 
\begin{equation}
\zeta(s) = \prod_{p} (1-p^{-s})^{-1}, 
\end{equation}
as a template, it is possible to define the zeta function of an algebraic curve over a finite field as follows.

\begin{definition}
Let $Y$ be a smooth, geometrically connected curve over $\mathbb{F}_q$, the finite field with $q$ elements. Then the zeta function of $Y$ is given by 
\begin{equation}
\zeta_Y(s) = \prod_{y \in |Y|} (1-|k_y|^{-s})^{-1},
\end{equation}
where $|Y|$ is the set of closed points of $Y$ and $|k_y|$ is the size of the residue field of $y$.  
\end{definition}

This formulation of the zeta function of an algebraic curve over a finite field shows the analogy with Riemann's zeta function, but we will prefer the following equivalent expression $\zeta_Y(s) = \text{exp} \left ( \sum_{n=1}^\infty {\frac{Y_n}{n} q^{-ns}} \right )$, where $Y_n$ is the number of points of $Y$ defined over $\mathbb{F}_{q^n}$, the degree $n$ extension of $\mathbb{F}_q$. The study of these zeta functions began in 1924 in Emil Artin's PhD thesis, \cite{Art}. These were further studied by F. K. Schmidt, who proved, in 1931, that $\zeta_Y(s)$ was a rational function of $q^{-s}$ in \cite{Sch}, and H. Hasse, who showed, in 1934, in \cite{Has}, that if $Y$ is an elliptic curve, then the zeros of $\zeta_Y(s)$ satisfy $\Re(s) = \frac{1}{2}$. Thus, the corresponding version of (RH) holds for these zeta functions of elliptic curves over finite fields. Furthering this idea, A. Weil then proved, in 1946--1948, that this same version of (RH) holds for algebraic curves of arbitrary genus and for abelian varieties in \cite{Wei1}. (See also \cite{Wei2}, \cite{Wei3} and \cite{Wei4}.) Below we present a sketch of some of the ideas contained in a modern proof of these results, which are based on Weil's ideas, and will motivate the work contained in this paper. 

First, a sequence of so-called "Weil cohomology" groups for the curve $Y$ are formed, in particular $H^0, H^1, H^2$ are the only nontrivial groups, with $\dim H^0 = \dim H^2 = 1$ and $\dim H^1 = 2g$ where $g$ denotes the genus of $Y$. Then the Frobenius map $F$ which sends $y \to y^q$ acts on the space $\mathbb{F}_{q^n}$ for any $n$ and therefore induces a morphism of the curve $Y$ over $\overline{\mathbb{F}_q}$ (the algebraic closure of $\mathbb{F}_q$) as well as, in fact, also induces a linear map on the cohomology groups $F^*: H^j \to H^j$, for $j \in \{0,1,2\}$. 

Next, consider the Lefschetz fixed point formula from topology.

\begin{thm}\label{lefthm}
\emph{(Lefschetz Fixed Point Formula)} Let $Y$ be a closed smooth manifold and let $f: Y \to Y$ be a smooth map with all fixed points nondegenerate. Then $\sum_{j=0}^\infty (-1)^j \text{Tr}(f^*| H^j)$ is equal to the number of fixed points of $f$. 
\end{thm}

Note that in Theorem \ref{lefthm}, since $Y$ is finite-dimensional, only finitely many of the cohomology spaces $H^j$ are nontrivial.

We apply the topological version of this result to the $n^{th}$ power of the Frobenius map, $F^n$, whose fixed points are exactly the points on the curve $Y$ defined over $\mathbb{F}_{q^n}$. That is, all those points with every coordinate in $\mathbb{F}_{q^n}$. This gives
\begin{equation}
\sum_{j=0}^2 (-1)^j \text{Tr}(F^{*^n}| H^j) = Y_n,
\end{equation}
where $F^{*}|H^j$ (for $j=0,1,2$) denotes the linear operator induced on the cohomology space $H^j$ by the Frobenius morphism $F$. 

To proceed further, we need the next result from linear algebra.

\begin{thm}
If $f$ is an endomorphism of a finite dimensional vector space $V$, then for $|t|$ sufficiently small, 
$\emph{exp} \left (\sum_{n=1}^\infty {\frac{1}{n} t^n \text{Tr}(f^n|V)} \right) = \det(I - f\cdot t|V)^{-1}$.
\end{thm}

Applying this result to the Frobenius operator $F$, we can proceed with the following calculation:
\begin{align}
\zeta_Y( s) & = \text{exp} \left (\sum_{n=1}^\infty {\frac{Y_n}{n} q^{-ns}} \right ) \nonumber \\
& = \text{exp} \left (\sum_{n=1}^\infty {\frac{1}{n} \sum_{j=0}^2 (-1)^{j} Tr(F^{*^n}| H^j) q^{-ns}} \right ) \nonumber \\
& = \prod_{j=0}^2 \left (\text{exp} \left (\sum_{n=1}^\infty {\frac{1}{n}  Tr(F^{*^n}| H^j) q^{-ns}} \right ) \right )^{(-1)^{j}} \nonumber \\
& = \prod_{j=0}^2 \left ( \det(I - F^* q^{-s} | H^j) \right )^{(-1)^{j+1}} \nonumber \\
& = \frac{\det(I - F^* q^{-s} | H^1)}{\det(I - F^* q^{-s} | H^0)\det(I - F^* q^{-s} | H^2)}.
\end{align}

This enables us to express the zeta function of a curve $Y$ as an alternating product of characteristic polynomials of the Frobenius operators, or more precisely, of determinants of $I - q^{-s} F^*$ over the cohomology spaces. Since these spaces are finite-dimensional, this equation further shows that $\zeta_Y(s)$ is a rational function of $q^{-s}$, which yields Schmidt's result. We also see that the zeros of $\zeta_Y(s)$ are given from the eigenvalues of the operator $F^*$ on $H^1$, while the poles are given from the eigenvalues on $H^0$ and $H^2$. To complete the proof, it was shown by Weil using the intersection theory of divisors to show that the intersection is positive definite that the eigenvalues of $F^*$ on $H^j$ have absolute value $q^{\frac{j}{2}}$ and thus the zeros of $\zeta(Y, s)$ satisfy $\Re(s) = \frac{1}{2}$. 

Weil then conjectured that all of the above and more could be generalized to any non-singular, projective variety of dimension $d$, defined over $\mathbb{F}_q$. About a decade later, Alexander Grothendieck announced he would be revamping algebraic geometry with the goal of proving these Weil conjectures. Several attempts to construct a proper "Weil cohomology" were incomplete, but eventually these provided the key idea to the proof of the Weil conjectures. Grothendieck even came up with more general conjectures based on this study of what properties a "Weil cohomology" must possess. The version of (RH) sought after would then follow from these. Some of his work outlining these ideas are \cite{Gro1}, \cite{Gro2}, and \cite{Gro3}. However, Pierre Deligne, a student of Grothendieck, would go on to prove, in 1973, this version of (RH) without proving Grothendieck's 'standard conjectures', which are still unproven today. See \cite{Del1} and \cite{Del2} for Deligne's work. Thus, Weil's conjectures were completed as a result of the introduction of, or at the very least, expansion of, the use of topology and cohomology in number theory. For a more complete history of the Weil conjectures, see \cite{Die}, \cite{Kat} and \cite{Oort}.

\subsection{Polya--Hilbert Operators and a Cohomology Theory in Characteristic Zero}

As seen in the previous section, the Frobenius operator became fundamental to the resolution of the version of (RH) dealing with algebraic varieties (or even with curves) over finite fields. The eigenvalues of this operator on different cohomology groups gave us the zeros and poles of the zeta function of the variety. If such an operator could be found for the Riemann zeta function, then perhaps this work would extend and help one to prove (RH). However, if you instead consider the function $\zeta(\frac{1}{2}+it)$ as a function of $t$, then (RH) is equivalent to all the nontrivial (or critical) zeros of this function being real. This then leads into what is known as the Polya--Hilbert conjecture.

\begin{myconjecture}
\emph{(Polya--Hilbert Conjecture)} The critical zeros of $\zeta(\frac{1}{2} +it)$ correspond to the eigenvalues of an unbounded self-adjoint operator on a suitable Hilbert space.
\end{myconjecture}

Since then, motivated in part by the above reformulation, many physicists, mathematicians and mathematical physicists have been looking for a convincing physical reason why (RH) should be true. In particular, it has been conjectured by Michael Berry in \cite{Ber} (and several other papers) that a trace formula for a suitable (classically chaotic) quantum-mechanical Hamiltonian could formalize this connection between the spectrum of an operator and the Riemann zeros. See also \cite{BerKea} for a discussion of these ideas. In fact, Alain Connes, in \cite{Conn}, conjectured the existence of a suitable noncommutative version of such a trace formula. However, as in every other approach to proving (RH), the search for the correct way to make this potential approach work continues to this day. 

Building on Alexander Grothendieck's ideas, Christopher Deninger has postulated in \cite{Den1}, \cite{Den2} (and other papers) that the cohomology theory used to prove the Weil conjectures could be extended to eventually prove the Riemann hypothesis. In particular, he envisions a cohomology theory of algebraic schemes over $Spec(\mathbb{Z})$ that would conjecturally help prove the Riemann hypothesis and solve other important problems in analytic number theory. In his papers, he lays out some of the difficulties in doing so as well as some of the properties that such a theory would need to satisfy. 

We also mention that Shai Haran \cite{Har2} has obtained interesting trace formulas yielding Weil's explicit formula: that is, of Weil's interpretation in \cite{Wei8} of Riemann's explicit formula (\cite{Rie}, \cite{Edw}, \cite{Tit}, \cite{Pat}, \cite{ISRZ}, \cite{LapvFr}).

\subsection{Fractal Cohomology}

All of the previous ideas as well as separate connections between fractals and the Riemann zeta function $\zeta(s)$ discussed in \cite{LapMai} and \cite{LapvFr}, motivated the second author to pursue a fractal cohomology to try to tie together all of the ideas presented so far. The text, \cite{ISRZ}, outlines his ideas for how the theory of fractals might give information about the Riemann zeta function. See also Section 12.4 of \cite{LapvFr} for a discussion of the main properties that such a fractal cohomology theory should satisfy, by analogy with the case of varieties over finite fields and self-similar strings. 

In search of the elusive 'Frobenius operator in characteristic 0', the second author worked with H. Herichi to develop a 'Quantized Number Theory' in \cite{HerLap1}, \cite{HerLap2}, \cite{HerLap3}, \cite{HerLap4}, \cite{Lap}. Here, they used an operator they denoted $\partial$, which was the derivative operator on a suitable family of Hilbert spaces. This operator had many nice properties, including being a generator for the infinitesimal shift group on these spaces as well as having a spectrum consisting of a single vertical line in the complex plane. This allowed them to focus on the values of $\zeta(s)$ on $\Re(s) = c$ for $c \in (0, \frac{1}{2})$ or for $c \in (\frac{1}{2}, 1)$ and obtain a reformulation of (RH) within this theory. This involved studying an operator-valued version of $\zeta(s)$, which they called a quantized zeta function. An overview of these ideas and results can be found in \cite{Lap}, while a detailed exposition of the theory is provided in \cite{HerLap1}. 

This paper then continues this search of an appropriate substitute for Frobenius in characteristic zero. In an attempt to further localize the spectrum of the derivative operator, we turn to a family of weighted Bergman spaces, which provide the basis for our construction. We will begin by recalling some needed functional analysis building up to the regularized determinants that we will need. Then we discuss the family of Bergman spaces and the needed properties of the derivative operator on them, which allows our construction to work. At this point, we will detail our construction to create a Frobenius replacement. This provides a general framework to find a substitute for the Frobenius operator which will be shown to apply to any entire function of finite order as well as certain meromorphic functions of interest such as $\zeta(s)$. However, there is still much to be done. We do not have a true cohomology theory as we do not have a suitable notion for how to define the geometry in our context. We will finish with a discussion of what is lacking from this theory and where to go from here.

%For example, the proof of the Weil conjectures outlined in the previous section relies on the Frobenius operator, but this proof has no known analogue for characteristic 0 situations such as the Riemann zeta function itself. How can we find an operator in characteristic 0 that will behave in some manner similarly to the Frobenius? This paper describes the work done by the authors in trying to answer this question

%This then leads to the further question: is there some other operator one can use, where the eigenvalues on certain cohomology spaces correspond to the zeros or poles of a given zeta function (or other meromorphic function) that would work in characteristic 0? Further, can we recover the zeta function as a determinant of $I-As$? Note that the fact there are infinitely many zeros or poles for some functions under consideration will necessitate considering determinants of operators over an infinite dimensional space, unlike the situation in the Weil conjecture where all of the cohomology spaces were finite dimensional. Deninger has suggested using zeta-regularized determinants to overcome the difficulty of determinants on infinite dimensional spaces. We choose to instead use the concepts of the Fredholm determinant and trace class operators as well as a different type of regularized determinants involving operators that are not trace class, but are instead in certain trace ideals. 

\section{Background}

This section loosely follows \cite{Sim} in developing the necessary theory for trace ideals and regularized determinants to be used in this paper. See \cite{Sim2, Sim} for detailed historical notes and references, along with a discussion of the many contributions to this subject.

To describe what trace ideals are, we recall some standard facts about compact operators on a separable Hilbert space $H$. 

\begin{thm}
Let $A$ be a compact operator on $H$. Then there are orthonormal sets $\{\psi_n\}$ and $\{\phi_n\}$ and positive real numbers $\mu_n(A)$, with $\mu_1(A) \geq \mu_2(A) \geq \cdots$, such that $ A = \sum_n \mu_n(A) (\psi_n, \cdot) \phi_n$. Moreover, the numbers $\mu_n(A)$ are uniquely determined.
\end{thm}

The positive real numbers $\mu_n(A)$ from the previous theorem are called the \emph{singular values} of $A$. We can actually describe $\{\mu_n(A)\}$ in another way. Given an operator $A$, the operator $A^* A$ is a nonnegative operator, so that $|A| := \sqrt{A^* A}$ makes sense. The $\mu_n(A)'s$ are exactly the (nonzero) eigenvalues of $|A|$. We can now turn to Calkin's theory of operator ideals. We begin by setting up a relationship between ideals in $B(H)$ and certain sequence spaces. 

\begin{definition}
Fix an orthonormal set $\{\phi_n\}$ in $H$. Given an ideal $J \neq B(H)$; we define the \emph{sequence space associated to} $J$ by 
\begin{equation}
\displaystyle S(J) = \{ a = (a_1, a_2, ...) | \sum_n a_n (\phi_n, \cdot) \phi_n \in J\}. 
\end{equation}
On the other hand, given a sequence space $\textbf{s}$, let $I(\textbf{s})$ be the family of compact operators $A$ with $(\mu_1(A), \mu_2(A), ...) \in \textbf{s}$.
\end{definition}

In order for this correspondence between sequence spaces and ideals to be one-to-one, we need to restrict our sequence spaces to Calkin spaces. We then need the following operator on sequences.

\begin{definition}
Given an infinite sequence, $(a_n)$, of numbers with $a_n \to 0$ as $n \to \infty$, $a_n^*$ is the sequence defined by $a_1^* = \max_i |a_i|$, $a_1^* + a_2^* = \max_{i \neq j} (|a_i| + |a_j|)$, etc. Thus $a_1^* \geq a_2^* \geq \cdots$, and the sets of $a_i^*$ and $|a_i|$ are identical, counting multiplicities. 
\end{definition}

\noindent This operator allows us to make the following definition.

\begin{definition}
A \emph{Calkin space} is a vector space, $\textbf{s}$, of sequences $(a_n)$ with $\displaystyle \lim_{n \to \infty} a_n = 0$, and the so-called Calkin property: $a \in \textbf{s}$ and $b_n^* \leq a_n^*$ implies $b \in \textbf{s}$.
\end{definition}

With these definitions in mind, we can use the following theorem to see a relation between two-sided ideals and Calkin spaces. 

\begin{thm}
\emph{\cite{Sim}} If $\textbf{s}$ is a Calkin space, then $I(\textbf{s})$ is a two-sided ideal of operators and $S(I(\textbf{s})) = \textbf{s}$. Furthermore, if $J$ is a two-sided ideal, then $S(J)$ is a Calkin space and $I(S(J)) = J$.
\end{thm}

We will now use this relation to define the ideals in the space of compact operators that we will be working with. 

\begin{definition}\label{Jpdef}
A compact operator $A$ is said to be in the \emph{trace ideal} $J_p$, for some $p \geq 1$, if $ \sum_n \mu_n(A)^p < \infty$. That is, $J_p$ is the ideal which is associated to the Calkin space $l^p$. An element $A$ of $J_1$ is called a \emph{trace class operator}. For $A \in J_1$, we define $ \text{Tr}(A) = \sum_n (\phi_n, A \phi_n)$ for any choice of orthonormal basis $\{\phi_n\}$. If $A \in J_2$, then we say that $A$ is \emph{Hilbert--Schmidt}.
\end{definition}

Trace class operators, $A$, are precisely those operators for which the expression $ \text{Tr}(A) = \sum_n (\phi_n, A \phi_n)$ is absolutely convergent and independent of the choice of orthonormal basis. Similarly, Hilbert--Schmidt operators are those for which $ \sum_n (A \phi_n, A \phi_n) = \lVert A \phi_n \rVert^2$ is convergent and independent of the choice of orthonormal basis. If $A$ is a trace class operator, then there is a method to define a so-called \emph{Fredholm determinant}, $\det(I+zA)$, which defines an entire function on $\mathbb{C}$. Operators of the form $I + zA$ for a trace class operator $A$ are called \emph{Fredholm}. This determinant can be defined in several equivalent ways. We list them here for trace class $A$ and $z \in \mathbb{C}$: \\
\begin{equation}
\det(I+zA) := e^{\text{Tr}(\log(I+zA))}
\end{equation}
for small $|z|$ and then analytically continued to the whole complex plane, 
\begin{equation}
\det(I + zA) = \displaystyle\sum_{k=0}^\infty z^k\text{Tr}(\wedge^k(A))
\end{equation}
with $\wedge^k(A)$ defined in terms of alternating algebras, and
\begin{equation}\label{deteqn}
\det(I + zA) = \displaystyle\prod_{k=1}^{N(A)} (1+z \lambda_k(A)),
\end{equation}
where the complex numbers $\lambda_k(A)$ are the nonzero eigenvalues of $A$ and $N(A)$ is the number of such eigenvalues, which can be infinite. In the latter case, the corresponding infinite product is convergent. 

A discussion concerning which of the above equations should be taken as a definition and which are to be proven appears briefly in Chapter 3 of \cite{Sim} and in more detail in \cite{Sim2}. For the work here, (\ref{deteqn}) will be the most convenient choice. One thing to note at this time though is that $\det(I+zA)$ does define an entire function by any of the above definitions, when $A$ is trace class. This then shows why one cannot hope to recover a meromorphic function by simply taking the determinant of a suitable operator without taking the quotient of such determinants as was seen in the discussion of the Weil conjectures.

Although some of the operators we will consider will not be trace class, they will at least be in one of the other trace ideals $J_n$, for some $n \in \mathbb{N}$. In this case, we can define a regularized determinant that will allow us to get a determinant formula for the operator. We start by considering an expression of the form $\det(I+zA) e^{-zTr(A)}$. For trace class operators $A$, both $\det(I+zA)$ and $e^{-zTr(A)}$ are convergent, but for Hilbert--Schmidt operators neither is necessarily well defined. And yet, when you consider the two factors together as a possibly infinite product over the eigenvalues of $A$, 
\begin{equation*}
\displaystyle\prod_{k=1}^{N(A)} \left ( (1+z\lambda_k(A))\exp ( -\lambda_k(A) z  \right ),
\end{equation*}
the combined term does converge for Hilbert--Schmidt operators. This idea can in fact be extended to get a convergent infinite product expression for operators in any $J_n$, which will be called the \emph{regularized determinant} of $A$. First we need a lemma.

\begin{lemma}
\emph{\cite{Sim}} For $A \in B(H)$, let 
\begin{equation}
R_n(A) = \left [ (I+A) \exp \left ( \sum_{j=1}^{n-1} {(-1)^j j^{-1}A^j} \right ) \right ] - I. 
\end{equation}
Then if $A \in J_n$, we have $R_n (A) \in J_1$.  
\end{lemma}

This associates a trace class operator to any given $A \in J_n$ and allows us to define the \emph{regularized determinant} of $A$ as follows:

\begin{definition}\label{detndef}
\cite{Sim} For $A \in J_n$, define $\det_n(I + A) = \det(I + R_n(A))$.
\end{definition}

Note that this definition implies that $\det_1(I+A) = \det(I+A)$, the usual Fredholm determinant. We will use these two notations interchangeably from here on. Also with this definition, we can now give a very similar product formula for the regularized determinant of a Hilbert--Schmidt operator, with each term having an exponential factor to help convergence along with some other interesting properties. This corresponds to the $n=2$ case of the following result.

\begin{thm}\label{detnthm}
\emph{\cite{Sim}} For $A \in J_n$, we have
\begin{equation}
\textstyle \det_n(I + \mu A) = \displaystyle\prod_{k=1}^{N(A)} \left [ (1+\mu\lambda_k(A))\exp \left (\sum_{j=1}^{n-1} (-1)^j j^{-1} \lambda_k(A)^j \mu^j \right ) \right ].
\end{equation}
\end{thm}

These regularized determinants are related to the usual Fredholm determinant of $1+A$ for trace class operators $A$ in the following fashion.

\begin{thm}
\emph{\cite{Sim}} For $A \in J_1$, we have
\begin{equation}
\textstyle \det_n(I + \mu A) = \det(I+A) \displaystyle \exp \left ({\sum_{j=1}^{n-1} {(-1)^j j^{-1} \text{Tr}(A^j)}} \right ).
\end{equation}
\end{thm} 

These ``regularized determinants", just as the Fredholm determinants, define an entire function and will be key to the precise formulation of our results. This will be because our construction will not always create a trace class operator for which the standard Fredholm determinant would apply. This will be the case, in particular, for the Riemann zeta function, for which the regularized determinant $\det_2$ will be needed; see Theorems \ref{xithm} and \ref{zetathm} in Section 5.4 below. However, we will show that for any entire function of finite order and for many meromorphic functions, our construction will give an operator that is at least in some $J_p$ and thus the regularized determinant will apply to it. 
\section{Derivative Operator on Weighted Bergman Spaces}

The search for an operator to possibly take the place of the Frobenius in the proof of the Weil conjectures led us to consider the derivative operator. A treatment examining the derivative operator on $L^2(\mathbb{R}, e^{-2ct}dt)$ and its use to create a \lq quantized number theory' can be found in the research monograph \cite{HerLap1}, as well as in the accompanying articles \cite{HerLap2}, \cite{HerLap3}, \cite{HerLap4} and \cite{Lap}. 

This paper takes a different direction with the derivative operator. We begin by following the treatment in \cite{AtzBri} in developing properties of the derivative operator on a certain family of weighted Bergman spaces. We will then continue beyond their results and use all of this to create an operator that might properly take the place of the Frobenius. We begin by recalling the definitions of the spaces we will be working with. (See, e.g., \cite{Hed}, for a general reference about Bergman spaces.)

\begin{definition}
We define a \emph{weight function} to be a positive continuous function $w$ on $\mathbb{C}$. Then, for $1 \leq p \leq \infty$, we define the \emph{weighted $L^p$ spaces} to be $L_w^p(\mathbb{C})$, the space of functions on $\mathbb{C}$ such that $fw \in L^p(\mathbb{C}, d\lambda)$, where $\lambda$ is the Lebesgue measure on $\mathbb{R}^2$, and equipped with the norm $\lVert f \rVert_{L^p_w} = \lVert fw \rVert_{L^p(\mathbb{R}^2)}$. Next, denote by $B_w^p$ the subspace of entire functions in $L_w^p$; then, $B_w^p$ is called a \emph{weighted Bergman space of entire functions}. 
\end{definition}

Note that the convention above for functions $f \in L^p_w$ would be those for which $\int_{\mathbb{C}}{|f|^pw^p d\lambda}<\infty$ instead of $\int_{\mathbb{C}}{|f|^p w d\lambda} < \infty$. Then we have the following basic fact about these spaces.

\begin{thm}
For $p \geq 1$, $B_w^p$ is a closed subspace of $L_w^p$ and hence is a Banach space. Also, for $p=2$, $B_w^2$ is a Hilbert space.
\end{thm}

Now we consider the differential operator $D = \frac{d}{dz}$ on the space $B_w^p$ and examine its properties; including for particular choices of $w$ and $p$.

Consider the following types of weight functions: $\displaystyle w(z) = e^{-\phi(|z|)}$, where $\phi$ is a nonnegative, concave, monotone (i.e., nondecreasing), subadditive function on $\mathbb{R}_+ = [0,\infty)$ such that $w(0) = 0$ and 
\begin{equation}
 \lim_{t \to +\infty} \frac{\phi(t)}{\log t} = +\infty.
\end{equation}
We next define
\begin{equation} \label{defa}
a = \lim_{t\to+\infty} \frac{\phi(t)}{t}.
\end{equation} 

We then have the following results in this situation (with $\mathbb{N}_0 := \{0,1,2,...\}$).

\begin{thm} \label{theorem1}
\emph{\cite{AtzBri}} Let $1 \leq p \leq \infty$ and $w$ be a weight function with constant $a$, as in (\ref{defa}) above. Then, \\
1) The differentiation operator $D = \frac{d}{dz}$ is a bounded linear operator on $B_w^p$; \\
2) For all $r > 0$, and for $n \in \mathbb{N}_0$, we have the bound $\lVert D^n \rVert \leq n! r^{-n} e^{\phi(r)}$. \\
\end{thm}
\begin{proof}
We will prove 2) noting that 1) follows from it. Suppose that $f \in B_p^w$ and  $r > 0$. Cauchy's formula for the $n^{th}$ derivative of $f$ reads as
\begin{equation} \label{CIF}
D^n f(z_0) = \frac{n!}{2\pi i} \int_{|z|=r} {\frac{f(z_0+z)}{z^{n+1}} dz}.
\end{equation}
We now consider the case $p=\infty$. Let $z_0, z \in \mathbb{C}$ with $|z| = r$. Then since $\phi$ is subadditive and monotonic, we have: $\phi(|z_0+z|) \leq \phi(|z_0|+|z|) \leq \phi(|z_0|)+\phi(|z|)$. Also, $\lVert f \rVert_{\infty, w} =\sup_{z \in \mathbb{C}} |f(z)|e^{-\phi(|z|)} \geq |f(z_0+z)|e^{-\phi(|z_0+z|)}$ by the definition of the norm. This leads to: $|f(z_0+z)| \leq \lVert f \rVert_{\infty, w} e^{\phi(|z_0+z|)}\leq \lVert f \rVert_{\infty, w} e^{\phi(|z_0|)} e^{\phi(|z|)}$. Then by (\ref{CIF}), we have for any $z \in \mathbb{C}$ that
\begin{equation*}
|D^nf(z)| \leq n! r^{-n} \sup_{|z|=r} |f(z_0+z)| \leq n! r^{-n} \lVert f \rVert_{\infty, w} e^{\phi(|z|)} e^{\phi(r)}.
\end{equation*}
Thus $\lVert D^n f \rVert_{\infty, w} = \displaystyle \sup_{z \in \mathbb{C}} {|D^nf(z)|e^{-\phi(|z|)}} \leq n! r^{-n} e^{\phi(r)} \lVert f \rVert_{\infty, w}$. Therefore, we have that $\lVert D^n \rVert \leq n! r^{-n} e^{\phi(r)}$. Next, we turn to the case $1 \leq p < \infty$. Let $z \in \mathbb{C}$. Applying H\"{o}lder's inequality in (\ref{CIF}) yields
\begin{equation*}
|D^n f(z)| \leq \frac{n!}{(2\pi)^{\frac{1}{p}} r^n }\left ( \int_{0}^{2\pi} {|f(z+re^{i\theta})|^pd\theta}\right )^{\frac{1}{p}}.
\end{equation*}
This leads to
\begin{equation*}
\int_{\mathbb{C}} {|D^n f(z)|^p e^{-p\phi(|z|)} d\lambda(z)} \leq \frac{n!^p}{2\pi r^{pn}} \int_{0}^{2\pi} { \left ( \int_{\mathbb{C}} {|f(z+re^{i\theta})|^pe^{-p\phi(|z|)} d\lambda(z)} \right ) d\theta}.
\end{equation*}
By making a change of variable, we can rewrite this as 
\begin{equation} \label{Dn}
\int_{\mathbb{C}} {|D^n f(z)|^p e^{-p\phi(|z|)} d\lambda(z)} \leq \frac{n!^p}{2\pi r^{pn}} \int_{0}^{2\pi} { \left ( \int_{\mathbb{C}} {|f(z)|^pe^{-p\phi(|z-re^{i\theta}|)} d\lambda(z)} \right ) d\theta}.
\end{equation}
Using the triangle inequality for $\phi$, 
\begin{equation*}
|\phi(|z|) - \phi(|z-re^{i\theta}|)| \leq \phi(|re^{i\theta}|) = \phi(r),
\end{equation*}
in the inner integral on the right-hand side of (\ref{Dn}), we obtain that
\begin{align*}
\int_{\mathbb{C}} {|f(z)|^pe^{-p\phi(|z-re^{i\theta}|)} d\lambda(z)} & = \int_{\mathbb{C}} {|f(z)|^pe^{-p\phi(|z|)}e^{p(\phi(|z|)-\phi(|z-re^{i\theta}|))} d\lambda(z)} \\
& \leq e^{p\phi(r)} \int_{\mathbb{C}} {|f(z)|^pe^{-p\phi(|z|)} d\lambda(z)} \leq e^{p\phi(r)} \lVert f \rVert_{p, w}^p.
\end{align*}
Applying this estimate to (\ref{Dn}) then yields 
\begin{equation*}
\int_{\mathbb{C}} {|D^n f(z)|^p e^{-p\phi(|z|)} d\lambda(z)} \leq \frac{n!^p}{2\pi r^{pn}} \int_{0}^{2\pi} { e^{p\phi(r)} \lVert f \rVert_{p, w}^p d\theta} = \frac{n!^p e^{p\phi(r)}}{r^pn} \lVert f \rVert_{p, w}^p.
\end{equation*}
Thus $\lVert D^n f \rVert_{p, w} \leq n! r^{-n} e^{\phi(r)} \lVert f \rVert_{p, w}$, and it follows that $\lVert D^n \rVert \leq n! r^{-n} e^{\phi(r)}$. This concludes the proof of the theorem.  \qedsymbol
\end{proof}

We also have the following result about the spectrum, $\sigma(D)$, of $D$. (See, e.g., \cite{Rud} for a discussion of spectral theory and the functional calculus in this context.)

\begin{thm} \label{theorem2}
\emph{\cite{AtzBri}} Under the conditions of Theorem \ref{theorem1}, the spectrum $\sigma(D)$ is given for any $a \geq 0$ by 
\begin{equation}
\sigma(D) = \Delta_a : = \{z \in \mathbb{C}: |z| \leq a\}. 
\end{equation}
In particular, if $a=0$, then $\Delta_a = \{0\}.$
\end{thm}
\begin{proof}
Let $e_\lambda(z) = e^{\lambda z}$ for $\lambda \in \mathbb{C}$. Clearly, we have $De_\lambda = \lambda e_\lambda$ and so $e_\lambda$ is an eigenvector of the operator $D$ with eigenvalue $\lambda$, as long as $e_\lambda \in B_w^p$. However, if $|\lambda| < a$ and we write $z = re^{i\theta}$ and $\lambda = |\lambda|e^{i\beta}$, we have the following: 
\begin{equation}
|e_\lambda(z)e^{-\phi(|z|)}| = |e^{|\lambda|re^{i(\beta+\theta)}-\phi(r)}| = e^{|\lambda|r\cos(\beta+\theta) - \phi(r)} \leq e^{r \left (|\lambda|-\frac{\phi(r)}{r} \right )}.
\end{equation}
 But by (\ref{defa}), then this function is integrable, so $e_\lambda \in B_w^p$ for $|\lambda| < a$. Thus we have $\Delta_a \subseteq \sigma(D)$. To complete the proof we will show that $r(D)$, the spectral radius of $D$, satisfies the inequality $r(D) \leq a$. It suffices to show that $r(D) \leq a+\epsilon$ for any $\epsilon > 0$; so let us fix $\epsilon > 0$. Then again using (\ref{defa}), we see that there is $t_0 > 0$ such that $\phi(t) \leq (a+\epsilon) t$ for $t \geq t_0$. Thus by part 2 of Theorem \ref{theorem1} we have that $\lVert D^n \rVert \leq C n! r^{-n} e^{(a+\epsilon)r}$ for any $r > 0$, $n=1, 2, ...$, where $C$ is a constant depending only on $\epsilon$. Minimizing this expression with respect to $r$ yields the critical value $r = \frac{n}{a+\epsilon}$. Substituting this choice of $r$ gives $\lVert D^n \rVert \leq C \frac{n!e^n(a+\epsilon)^n}{n^n}$. Applying Stirling's formula gives that $\lVert D^n \rVert \leq f(n)$, where $f(n)$ is asymptotic to a constant times $\sqrt{n}(a+\epsilon)^n$. Thus we have $r(D) = \displaystyle \lim_{n \to \infty}  \lVert D^n \rVert^{\frac{1}{n}} \leq a+\epsilon$. We conclude that $r(D) = a$ and $\sigma(D) = \Delta_a$, as desired. \qedsymbol
\end{proof}

To further study this operator, we restrict our attention to the special case when $p=2$, where we actually have a Hilbert space with inner product given by $(f, g) = \int_{\mathbb{R}^2} {\overline{f}g e^{-2\phi(|z|)} d\lambda}$. It is convenient to have a particular simple orthonormal basis to deal with, and, since we are dealing with entire functions that are guaranteed to have convergent power series, it makes sense to look at polynomials to try to find this orthonormal basis. It turns out that all we need are monomials.

\begin{thm}
\emph{\cite{AtzBri}} There exist constants $c_n$ such that $\{u_n(z)\}$, where $u_n(z) = c_n z^n$, for $n \in \mathbb{N}_0$, forms an orthonormal basis for $B_w^2$. 
\end{thm}
\begin{proof}
First note that $(z^n, z^m) = 0$ if $n \neq m$. This follows from a simple calculation using the fact that the weight function is radial, so that using polar coordinates, we have 
\begin{align*}
(z^n, z^m) & = \int_0^{2\pi} {\int_0^{\infty} {r^n e^{-in\theta} r^m e^{im\theta} e^{-2\phi(r)} rdrd\theta}} \\
& = \int_0^{2\pi} {e^{i(m-n)\theta }d\theta} \int_0^\infty {r^{n+m+1}e^{-2\phi(r)}dr} \\
& = 0
\end{align*}
 if $n \neq m$. Note that the integral on $r$ converges for any $n, m \in \mathbb{N}_0$ by the properties of our weight function. Thus the monomials form an orthogonal set. This orthogonal set is complete because every entire function has a convergent power series on $\mathbb{C}$. Thus, if we choose $c_n = \frac{1}{\lVert z^n \rVert}$, we normalize our set and, hence, the resulting sequence $\{u_n\}$ is an orthonormal basis for $B_w^2$. \qedsymbol
\end{proof}

Now we specialize further by choosing the family of weight functions given by $w(z) = e^{-|z|^\alpha}$ for $\alpha\in\mathbb{R}$ with $0 < \alpha \leq 1$. (Note that in the notation of (\ref{defa}) and of Theorem \ref{theorem2} above, we then have $a = 1$ if $\alpha = 1$ and $a=0$ if $0 < \alpha < 1$.) We will call the resulting Hilbert space $H_\alpha := B_w^2$. In this case, we can actually find the constants $c_n$ explicitly. 

\begin{thm}
\emph{\cite{AtzBri}} If $0 < \alpha \leq 1$, then for $n \in \mathbb{N}_0$, we have that
\begin{equation*}
\lVert z^n \rVert_{H_\alpha}^2 = \frac{2\pi}{\alpha} 2^{-\frac{2}{\alpha}(n+1)} \Gamma\left [\frac{2}{\alpha}(n+1)\right ].
\end{equation*}
\end{thm}
\begin{proof}
Computing the norm in $H_\alpha$ gives
\begin{equation*}
\lVert z^n \rVert_{H_\alpha}^2 = \int_{\mathbb{C}} {|z|^n e^{-2r^{\alpha}}dz} = \int_0^{2\pi} {1d\theta} \int_0^\infty {r^{2n+1} e^{-2r^\alpha} dr}.
\end{equation*}
For the integral over $r$, we make the change of variable $x = 2r^\alpha$, which changes the integral into
\begin{equation*}
\lVert z^n \rVert_{H_\alpha}^2 = 2\pi \int_0^{\infty} {\left (\left (\frac{x}{2}\right )^{\frac{1}{\alpha}} \right )^{2n+1} e^{-x} \frac{1}{2\alpha(\frac{x}{2})^{\frac{\alpha-1}{\alpha}}} dx} = \frac{2\pi}{\alpha} 2^{-\frac{2}{\alpha}(n+1)} \int_0^\infty { x^{\frac{2(n+1)}{\alpha}-1}e^{-x}dx}.
\end{equation*}
However, the final integral is simply $\Gamma(\frac{2}{\alpha}(n+1))$. \qedsymbol
\end{proof}

Thus we can simply take the normalizing constants $c_n$ to be the square root of the reciprocal of the formula for $\lVert z^n \rVert^2$ given above. Namely, 
\begin{equation*}
c_n = \left ( \frac{2\pi}{\alpha} \right )^{-\frac{1}{2}} 2^{\frac{n+1}{\alpha}} \left (\Gamma \left (\frac{2}{\alpha} (n+1) \right ) \right )^{-\frac{1}{2}}
\end{equation*}
for every $n \in \mathbb{N}_0$. Now examining the action of $D$ on a typical basis element $u_n(z)$, we see that: $Du_n = D(c_n z^n) = n c_n z^{n-1} = \frac{n c_n}{c_{n-1}} u_{n-1}$. We thus obtain the following representation of $D$:

\begin{thm}
The operator $D$ is isomorphic to a weighted backward shift on $H_\alpha$ taking a sequence of coefficients $(a_n)$ in $l^2(\mathbb{C})$, where $f(z) = \sum_{n=0}^\infty a_n u_n(z)$, to $(\gamma_n a_{n+1})$ in $l^2(\mathbb{C})$, where for $n \in \mathbb{N}_0$, $\gamma_n > 0$ and $\gamma_n$ is given by $\gamma_n^2 = 2^{\frac{2}{\alpha}} \frac{(n+1)^2\Gamma(\frac{2}{\alpha}(n+1))}{\Gamma(\frac{2}{\alpha}(n+2))}$.
\end{thm}
\begin{proof}
Using the last calculation and writing $f(z) = \sum_{n=0}^\infty a_n u_n$, we obtain:
\begin{equation*}
D f(z) = \sum_{n=0}^\infty a_n \frac{nc_n}{c_{n-1}} u_{n-1} = \sum_{n=0}^\infty  \frac{(n+1)c_{n+1}}{c_{n}} a_{n+1} u_n = \sum_{n=0}^\infty  \gamma_n a_{n+1} u_n,
\end{equation*}
where $\gamma_n = \frac{(n+1)c_{n+1}}{c_{n}}$. It follows, using the previously calculated formula for $c_n$, that 
\begin{equation*}
\gamma_n^2 = \frac{(n+1)^2 c_{n+1}^2}{c_n^2} = \frac{(n+1)^2 \lVert z^n \rVert^2}{\lVert z^{n+1} \rVert^2} = 2^{\frac{2}{\alpha}} \frac{(n+1)^2\Gamma(\frac{2}{\alpha}(n+1))}{\Gamma(\frac{2}{\alpha}(n+2))},
\end{equation*}
as desired. \qedsymbol
\end{proof}

The last fact we will need from \cite{AtzBri} is to apply the standard asymptotic for the Gamma function to obtain that $\gamma_n \sim c \cdot n^{1-\frac{1}{\alpha}}$ as $n \to \infty$, where $c$ is a positive constant. Thus if $0 < \alpha < 1$, then $\gamma_n \to 0$ as $n \to \infty$. 

Continuing beyond the results from \cite{AtzBri}, we start by calculating the adjoint $D^*$. 

\begin{thm}
Given $f \in H_\alpha$, let $ f = \sum_{n=0}^\infty a_n u_n$ be its expansion in terms of the orthonormal basis. The adjoint of $D^*$ is isomorphic to a weighted forward shift given by the equation $D^* (a_n) = \left ( \gamma_{n-1} a_{n-1} \right )$.
\end{thm}
\begin{proof}
To calculate $D^*$, write $ D^*f = \sum_{n=0}^\infty b_n u_n$. Since $\{u_n\}$ is an orthonormal basis we find the $n^{th}$, for $n \geq 1$, coefficient of $D^*f$: $(D^*f, u_n) = (f, Du_n) = (f, \gamma_{n-1} u_{n-1}) = \gamma_{n-1} a_{n-1}$. Thus we have $b_n = \gamma_{n-1} a_{n-1}$ for each $n\geq 1$. For $b_0$, we calculate $(D^*f, u_0) = (f, Du_0) = (f, 0) = 0$. Thus, $D^*$ acts on the sequence of coefficients $(a_n)$ as a weighted forward shift $(a_n) \mapsto (\gamma_{n-1} a_{n-1})$, with the new $0^{th}$ coefficient being 0. \qedsymbol
\end{proof}

Now that we have the adjoint, we can immediately see that $D$ is not self-adjoint as $D$ is a backward shift and $D^*$ is a forward shift. Moreover, the following calculation with $f(z) \equiv 1$ shows that it is not even normal: Indeed, $D^*Df = D^* 0 = 0$, but, on the other hand, $DD^*f = D (\gamma_0 z) = \gamma_0$. This shows that we cannot apply the functional calculus for unbounded normal operators that was used in \cite{HerLap2}--\cite{HerLap1}. Instead we use the Riesz functional calculus, which is valid for bounded operators like $D$. 

Next, we will use the asymptotic $\gamma_n \sim c \cdot n^{1-\frac{1}{\alpha}}$ to determine which trace ideals $D$ will belong to, depending on $\alpha$.

\begin{thm}
The operator $D$ is compact on $H_\alpha$ for any $0 < \alpha < 1$, trace class for any $0 < \alpha < \frac{1}{2}$, Hilbert--Schmidt for any $0 < \alpha < \frac{2}{3}$, and, in general, $D \in J_p$ if $\alpha < \frac{p}{p+1}$ for any $p \in \mathbb{N}$, where $J_p$ is the trace ideal defined in the previous section.
\end{thm}
\begin{proof}
Let $0 < \alpha < 1$, and let $E_N : H_\alpha \to H_\alpha$ that takes a power series $ \sum_{n=1}^\infty a_n u_n \mapsto \sum_{n=1}^N a_n \lambda_{n-1} u_{n-1}$, which is the composition $DP_N$ of the derivative operator $D$ with the projection onto the subspace of polynomials at most degree $N$, $P_N$. Each $E_N$ is of finite rank, in fact, the range of $E_N$ has dimension $N$. We claim that the norm limit of $E_N$ is $D = \frac{d}{dz}$. Note that $\lVert D - E_N \rVert = \sup_{n > N} \{ \lambda_{n-1}\}$. But $\displaystyle \lambda_{n}  \sim c\cdot n^{1-\frac{1}{\alpha}} \to 0$, as $n \to \infty$, for any $0 < \alpha < 1$. Thus $E_N$ converges to $D$ in norm and therefore $D$ is compact. Furthermore, we can write $ D = \sum_{n=1}^\infty \lambda_{n-1} (u_n, \cdot) u_{n-1}$ so $\lambda_{n-1}^*$ are the singular values of $D$. To determine when $\lambda_{n-1}^*$ are in $l^p$, we use the Limit Comparison Test to compare $ \sum_{n=1}^\infty (\lambda_{n-1}^*)^p$ with $ \sum_{n=1}^\infty {(n^{1-\frac{1}{\alpha}})^p}$, which converges if and only if $p\left (1 - \frac{1}{\alpha} \right) < -1$. Solving this gives $\alpha < \frac{p}{p+1}$. Therefore $D \in J_p$ if $\alpha < \frac{p}{p+1}$ and, in particular, is trace class if $p < \frac{1}{2}$ and Hilbert--Schmidt if $p < \frac{2}{3}$. (Here, we have used the notation of Definition \ref{Jpdef}.) \qedsymbol
\end{proof}

\emph{From now on, we will fix an }$\alpha$\emph{ with }$0 < \alpha < \frac{1}{2}$\emph{ and simply refer to }$H_\alpha$\emph{ as }$H$. In this case, we have the following spectrum for $D$.

\begin{thm}\label{Dtheorem}
We have $\sigma(D) = \sigma_p(D) = \{0\}$ and 0 is a simple eigenvalue of $D$ with eigenfunction $f(z) \equiv 1$, the constant function equal to 1. 
\end{thm}
\begin{proof}
We know from Theorem \ref{theorem2} that $\sigma(D) = \Delta_a$, where $a = \displaystyle \lim_{t \to \infty} \frac{\phi(t)}{t}$, as in (\ref{defa}). Here we have $\phi(t) = t^\alpha$ for $0 < \alpha < \frac{1}{2}$. Thus we have that $\displaystyle \lim_{t \to \infty} \frac{t^\alpha}{t} = \lim_{t \to \infty} t^{\alpha - 1} = 0$. It follows that $a = 0$ and hence, by Theorem \ref{theorem2}, $\sigma(D) = \Delta_0 = \{0\}$. However, we also know that $f(z) \equiv 1 \in H$, so that $D$ has the eigenvector $f$ corresponding to the eigenvalue $0$ and the point spectrum of $D$ is also $\sigma_p(D) = \{0\}$. \qedsymbol
\end{proof}

Finally, we turn to considering the set of operators $\{e^{-s D}\}_{s \in \mathbb{C}}$. We compare this to the result for $\partial_c$ obtained in \cite{HerLap1} and mentioned in Section 1. This theorem will show that $D$ is the \emph{infinitesimal shift} (of the complex plane). 

\begin{thm}
The family $\{e^{-sD}\}_{s\in \mathbb{C}}$ gives the group of translation operators on $H$. 
\end{thm}
\begin{proof}
First note that since any $f \in H$ is an entire function, we have the convergent power series representation: $f(z - s) = \sum_{n=0}^\infty {\frac{f^{(n)}(z)}{n!}(-s)^n}$ for any $z, s \in \mathbb{C}$. Thus 
\begin{equation*}
e^{-sD} f(z) = \sum_{n=0}^\infty \frac{1}{n!} (-sD)^n f(z) = \sum_{n=0}^\infty \frac{1}{n!} (-s)^n \frac{d^n}{dz^n}f(z) = f(z-s).
\end{equation*}
This shows that $e^{-sD}$ just acts as translation by $s$ on the space $H$. From this expression we also see that $\displaystyle \lim_{s \to 0} \lVert e^{-sD} f - f \rVert = 0$. \qedsymbol
\end{proof}

\section{The Construction}

With the results about $D$ in hand, we turn to constructing an operator that might play the role of Frobenius when dealing with the Riemann zeta function or other entire or meromorphic functions of interest in number theory, analysis or mathematical physics. 

We begin by considering a particular choice of the family of analytic functions $\phi_\tau(z) = z+\tau$. This gives us operators 
\begin{equation}
D_\tau := \phi_\tau(D) = D + \tau I,
\end{equation}
for which the following lemma holds. (Recall that $\alpha$ has been fixed once for all to satisfy $0 < \alpha < \frac{1}{2}$ and hence, that Theorem \ref{Dtheorem} applies.)

\begin{lemma}
For any $\tau \in \mathbb{C}$, $D_\tau \in B(H)$ with spectrum $\sigma(D_\tau) = \{\tau\}$. If $\tau \neq 0$, then $D_\tau$ is invertible and $D_\tau^{-1}\in B(H)$. 
\end{lemma}
\begin{proof}
Applying the functional calculus on bounded operators along with the Spectral Mapping Theorem to the operator $D$ and the function $\phi_\tau (z)$ gives a bounded operator $D_\tau$ with spectrum $\sigma(D_\tau) = \phi_\tau(\{0\}) = \{\tau\}$, where we have used Theorem \ref{Dtheorem} according to which $\sigma(D) = \{0\}$. Furthermore, if $\tau \neq 0$, then $0 \notin \sigma(D_\tau)$ and it follows that $D_\tau$ has a bounded inverse. \qedsymbol
\end{proof}

This gives us a family of operators, each of whose spectra are each a single point, which can be any complex number. Recall that in the situation of the cohomology theory that helped prove the Weil conjectures, we would like an operator whose eigenvalues on different cohomology spaces are the zeros and poles of the zeta function we are interested in. In order to obtain an operator whose spectrum can represent the zero or pole set of a meromorphic function, we use the following construction. If $Z = \{z_1, z_2,...\}$ is a (finite or countable) multiset of complex numbers, let $H_n$ be a copy of the weighted Bergman space $H$ and associate an operator $D_n$ to be $D_{z_n}$ on $H_n$. (Here and thereafter, a multiset is a set with integer multiplicities.) Finally, define the Hilbert space $H_{Z} = \bigoplus_{n} H_n$ with operator $D_{Z} = \bigoplus_{n} D_n$. This gives:

\begin{thm}\label{spectrumthm}
For $Z = \{z_1, z_2, ...\}$, the operator $D_{Z}$, constructed above, has spectrum $\sigma(D_{Z}) = \overline{\{z_1,z_2,...\}}$. Furthermore, for each $i \in \mathbb{N}$, $z_i$ is an eigenvalue of $D_Z$ and the number of linearly independent eigenvectors of $z_i$ for $D_{Z}$ in $H_{Z}$ is equal to the number of times $z_i$ occurs in the multiset $Z = \{z_1,z_2,...\}$. 
\end{thm}

\begin{proof}
For each $n \in \mathbb{N}$, let $e_n \in H_Z$ be the element which is the constant function, with value $1$ in the $n^{th}$ component, and $0$ in every other component. Then $D_Z e_n = z_n e_n$ and so $z_n$ is an eigenvalue with eigenvector $e_n$. Suppose $z_{n_1} = z_{n_2} = \cdots = z_{n_k} = z$. Then $z$ is an eigenvalue with eigenvectors $e_{n_1}, e_{n_2}, ... e_{n_k}$ and so there are at least as many linearly independent eigenvectors of $z$ for $D_Z$ as the multiplicity of $z$ in the multiset. Next, recall that the only eigenvalue of $\frac{d}{dz}$ on $H$ is $0$. Thus, the only eigenvalue of $D_{z_n}$ is $z_n$. Suppose now that $D_Z x = z x$ for some $x$, we must either have the $n^{th}$ component of $x$ being $0$ or $z = z_n$ and so there cannot be any more linearly independent eigenvectors of $z$ for $D_Z$. Now that we know $z_i$ is an eigenvalue of $D_Z$ for each $i$, we know that $\overline{\{z_1, z_2, ...\}} \subset \sigma(D_Z)$. Next, let $\lambda \in \mathbb{C} - \overline{\{z_1, z_2, ... \}}$. Then $d = \inf_{n \geq 0} |\lambda - z_n| > 0$. Since $D = \frac{d}{dz}$ is quasinilpotent, $r(D) = 0$ and so there is a positive integer $N$ such that for every integer $k \geq N$, we have $\lVert D^k \rVert < \left (\frac{d}{2} \right)^k$. Then, on the $n^{th}$ component of $H_Z$, we have that $\sum_{k=0}^\infty {\frac{ D^k }{|\lambda - z_n|}}$ is absolutely convergent because 
\begin{equation*}
\sum_{k=N}^\infty {\frac{\lVert D^k \rVert}{|\lambda - z_n|}} \leq \sum_{k=N}^\infty {\frac{\left ( \frac{d}{2} \right )^k}{d^k}} = \frac{1}{2^{N-1}}.
\end{equation*}
Then we can calculate the inverse on the $n^{th}$ component via the absolutely convergent series:
\begin{equation*}
(\lambda I - D_{z_n})^{-1} = ((\lambda - z_n)I - D)^{-1} = \frac{1}{\lambda - z_n} \sum_{k=0}^\infty \frac{ D^k }{\lambda - z_n}.
\end{equation*}
Further, by the same estimate $\lVert (\lambda I - D_{z_n})^{-1} \rVert \leq C$ uniformly in $n$, where 
\begin{equation}
\displaystyle C = {\sum_{k=0}^N \frac{\lVert D^k \rVert}{d^k}} + 2^{1-N}.
\end{equation}
Therefore, $ \bigoplus_{n} (\lambda I - D_{z_n})^{-1} \in B(H_Z)$ and so $(\lambda I - D_Z)^{-1}$ exists and is bounded. That is, $\lambda \in \rho(D_Z)$, the resolvent set of $D_Z$; recall that by definition, $\sigma(D_Z)$ is the complement of $\rho(D_Z)$ in $\mathbb{C}$. Hence, $\sigma(D_{Z}) = \overline{\{z_1,z_2,...\}}$. \qedsymbol
\end{proof}
\begin{corollary}
If $Z = \{z_1, z_2, ...\}$ is either the zero or pole set of a meromorphic function $f(z)$, then $Z$ is a discrete set and so we have exactly $\sigma(D_Z) = Z$ counting multiplicity. Moreover, each $z_i$ in $Z$ is an eigenvalue of $D_Z$, with multiplicity equal to the multiplicity of $z_i$ in the multiset $Z$.
\end{corollary}

Thus, $D_Z$ has all of the information from the multiset $\{z_1,z_2,...\}$ contained in its spectrum. If we then consider the multiset to be the zeros and poles of a meromorphic function $f(z)$, then the operator $D_Z$ contains these pieces of information of this function. This was, in fact, the original goal of this direction. In \cite{HerLap1} (also \cite{HerLap2}, \cite{HerLap3}, and \cite{HerLap4}) the spectrum of the operators studied there were vertical lines in the complex plane and the values of $\zeta(s)$ on these lines. This work was meant to approach the problem in a similar, but different, way and isolate out the zeros and poles of certain meromorphic functions. Theorem \ref{spectrumthm} gives a positive result that we have created such an operator. However, we wanted to go further and find a way to use determinant formulas to fully recover all of the values of the desired function as was done in the Weil conjectures. Unfortunately, we cannot use the determinant formulas for operators given in Section 2 for the operator $D_Z$ to recover $f(z)$ as a whole, because with this formulation $D_Z$ is not trace class. Even looking at just a single one of the terms, $D_\tau = D + \tau I$, with $\tau \neq 0$, is not compact, let alone trace class (or, consequently, in any of the operator ideals $J_p$), because our Hilbert space is infinite dimensional. (Indeed, in a Banach space, the identity operator, like the closed unit ball, is compact if and only if the space is finite-dimensional.)  

The first such modification we will make is to restrict each $D_\tau$ to its eigenspace, $E$, the space of constant functions. This restriction is then a compact operator. We also need to make a second adjustment from the original idea. Any zero or pole set of a meromorphic function will be a discrete set, and hence if there are infinitely many zeros or poles, they must tend to $\infty$. This would then imply that the operator $D_Z$ described here is unbounded. What allows us to repair this problem and simultaneously recover the given function $f(z)$ using determinants, is to have our set $Z$ consist of the reciprocals of the zeros rather than of the zeros themselves, and similarly for the poles. (In each case, the multiplicities of the zeros or the poles are taken into account.) The set of reciprocals will not necessarily be discrete as if the set is infinite, the sequence will tend to 0. However, rather than being a problem, this is actually completely required. Indeed, a compact operator on an infinite dimensional space cannot have a bounded inverse and so 0 must be in the spectrum of any such compact operator. Combining this observation with the fact that regularized determinants apply only to trace ideals of compact operators, we see that having 0 in the closure of the set of reciprocals is necessary to apply the determinant theory to $D_Z$.  

One comment to make about the restriction of the operator $D_Z$ to its total eigenspace, $E_Z$, is that it simplifies the operator to a multiplication operator because the derivative on constant functions is just the zero operator. This is unfortunate as we do lose some of the rich setting of the Bergman space that has been used thus far. We are currently exploring alternative constructions in \cite{CobLap} that would allow us to remove this restriction and work on all of $H$. However, as we will see in the next section, the new version of the operator obtained by restriction will retain the desired spectrum from Theorem \ref{spectrumthm}.

In addition, we observe that in some sense, by analogy with what happens for curves over finite fields for Weil-type cohomologies and with what is expected in more general situations associated with the Riemann zeta function and other $L$-functions (see, e.g., \cite{Den1, Die} and \cite[esp. Appendix B]{ISRZ}), $E_Z$ (the total eigenspace of $D_Z$) is the counterpart in our context of the total cohomology space (or, in the terminology of \cite{ISRZ, LapvFr}, the total ``fractal cohomology space'') and correspondingly, the restriction of the original (generalized) Polya--Hilbert operator $D_Z$ to its eigenspace $E_Z$ is the counterpart of the linear endomorphism induced by the Frobenius morphism on the (total) cohomology space. (See Section 1.1 above.) Therefore, this modification of the original operator seems natural (and perhaps unavoidable) in order to obtain a suitable determinant formula, of the type obtained in Section 4.1 and Section 5 below. 

%If it is trace class or higher classes
\subsection{Refining the Operator of a Meromorphic Function}

First, let $Z = (z_n)$ be a sequence of complex numbers. Let $D_n := D + z_n I$ be the operator in the previous section restricted to the subspace of constant functions $E$. (It is clear that $D_n$ is normal.) Let $D_Z = \bigoplus_n D_n$ act on the space $E_Z = \bigoplus_n E$ which is a closed subspace of the Hilbert space $H_Z$ from the last section. So in actuality, this new definition of $D_Z$ is just the restriction to the Hilbert space $E_Z$ of the operator given in the previous section. First we note that this restriction still retains the main property from the last section.

\begin{thm}
For each $n \in \mathbb{N}$, each $z_n$ is an eigenvalue of $D_Z$ and the number of linearly independent eigenvectors associated to $z_n$ is equal to the number of times $z_n$ occurs in the sequence $Z$. Furthermore, $\sigma(D_Z)=\overline{\{z_n : n=1,2,3,...\}}$.
\end{thm}
\begin{proof}
Let $e_n$ be the eigenfunctions from the previous proof. Then since $e_n$ is constant in each coordinate, $e_n \in E_Z$. Thus when restricted to the space of functions constant on each coordinate, $D_Z$ retains all of its eigenvalues and eigenvectors from before. Finally, we note that $\sigma(D_{z_n}) = \{z_n\}$ from which it follows that $\sigma(D_Z)=\overline{\{z_n : n=1,2,3,...\}}$ as in the proof of Theorem \ref{spectrumthm}. \qedsymbol
\end{proof}

\begin{remark}
Note that in the case when $Z = (z_n)$ is the sequence of the reciprocals of the nonzero elements in the zero set or the pole set of a given meromorphic function (as in Section 5 below), then $\sigma(D_Z) = Z$ if $Z$ is finite and $\sigma(D_Z) = Z \bigcup \{0\}$ if $Z$ is infinite.
\end{remark}

The next theorem shows that this restriction of the operator will truly give us what we need for our quantized number theory framework.

\begin{thm}\label{mainthm}
We have the following relationships between an infinite sequence $Z = (z_n)$ and the associated operator $D_Z$. \\
1) $D_Z$ is bounded iff $(z_n)$ is a bounded sequence.\\
2) $D_Z$ is self-adjoint iff $z_n \in \mathbb{R}$ for all $n$.\\
3) $D_Z$ is compact iff $\displaystyle \lim_{n \to \infty} z_n = 0$.\\
4) $D_Z$ is Hilbert--Schmidt iff $ \sum_{n=1}^\infty |z_n|^2 < \infty$. \\
5) $D_Z$ is trace class iff $ \sum_{n=1}^\infty |z_n| < \infty$. \\
6) For $p \geq 1$, $D_Z \in J_p$ iff $ \sum_{n=1}^\infty |z_n|^p < \infty$.\\
If $(z_n)$ is a finite sequence, then $D_Z$ is bounded, compact, and in $J_p$ for each $p \geq 1$.
\end{thm}
\begin{proof}
Since $\lVert D_n \rVert = |z_n|$, for each $n \in \mathbb{N}$, we have $\lVert D_Z \rVert = \sup_{n} |z_n|$. Then $D_Z$ is bounded iff $(z_n)$ is a bounded sequence. For 2), consider the sequence of operators $\overline{D}_N = \bigoplus_{n=1}^N D_n$, for $N \in \mathbb{N}$, as an operator on $E_Z$ by letting it act as multiplication by $0$ on the remaining components. Thus $\overline{D}_N$ is a finite rank operator on $E_Z$ for each $N$. Then $\lVert D_Z - \overline{D}_N \rVert = \sup_{k>N} |z_k|$ and so if $\lim_{n \to \infty} z_n = 0$, we then have that $D_Z$ is the norm limit of finite rank operators and thus is compact. On the other hand, if $\lim_{n \to \infty} z_n \neq 0$ then $\{e_n\}$ is a bounded sequence of vectors such that $\{ D_Z e_n\}$ has no convergent subsequence. Thus, $ D_Z$ is not compact. For 3) and 4), assume that $ D_Z$ is compact. Then since $ D_Z e_n = z_n e_n$ and the fact that $\{e_n\}$ forms an orthonormal basis for $E_Z$ we know the singular values of $E_Z$ are $\{z_n^*\}$, which is just the sequence of numbers $|z_n|$ arranged in nonincreasing order. Thus $E_Z$ is Hilbert--Schmidt iff $ \sum_{n=1}^\infty |z_n|^2 < \infty$, and, trace class iff $ \sum_{n=1}^\infty |z_n| < \infty$ and, more generally, in $J_p$ iff $ \sum_{n=1}^\infty |z_n|^p < \infty$. Finally, if $(z_n)$ is a finite sequence, then $D_Z$ is actually a finite rank operator and is trivially bounded, compact, and in $J_p$ for each $p \geq 1$. \qedsymbol
\end{proof}

Note that the above result is well known from the theory of multiplication operators on sequence spaces and is just translated here in our setting. Now that we have a formulation that can indeed give us a trace class operator, we can state the result we will use to fully recover certain functions of interest.

\begin{corollary}
If $\{z_n\}$ is a sequence of complex numbers satisfying $ \sum_{n=1}^\infty |z_n| < \infty$, then we have $\det(I - zD_Z) = \prod_n (1- z_n z)$.
\end{corollary}

\begin{proof}
This is a direct consequence of Equation (\ref{deteqn}) for trace class operators of which $D_Z$ is one when the series is absolutely summable. \qedsymbol
\end{proof}

By the previous corollary, we can now see that we will be getting an entire function out of our construction. Thus if we want to handle meromorphic functions, we will need to handle zeros and poles separately. Also, we will want to choose our sequence $(z_n)$ to be the reciprocals of the zeros and, separately, of the poles. With this in mind, we make the following final construction for our operator associated with a meromorphic function.

First, let $f(z)$ be an entire function on $\mathbb{C}$ with $z=0$ not a zero of $f$. Let $\{a_n\}$ be a sequence of the zeros of $f(z)$, counting multiplicity. Define the sequence $Z = (z_n)$ where $z_n = \frac{1}{a_n}$. Define $D_Z$ as before and call this $D_{f}$. Now given an integer $m \geq 1$, if we have $D_{f} \in J_m \setminus J_{m-1}$, then $\det_m (I - zD_Z)$ is well defined, where the regularized determinant $\det_m$ was defined in Section 2. (See, especially, Definition \ref{detndef} and Theorem \ref{detnthm}.) Finally, we note that if we are dealing with a meromorphic function instead of an entire function, we follow the lead from the proof of the Weil conjectures to simply take the ratio of these regularized determinants, with possibly the operator associated to the numerator being in different trace classes (that is, in different operator ideals) than that of the denominator. 

In the next section, we will examine what this construction accomplishes for several classically important functions in number theory.

\section{Applications of the Construction}

In this section, we apply our construction to some special meromorphic functions of interest and conclude with showing that this construction does indeed give a suitable replacement for the Frobenius for any entire function of finite order. 

\subsection{Rational Functions}

To begin, we look at the simplest type of meromorphic functions: the rational functions. Let $f(z)$ be a rational function. Then, we can write $f(z) = z^k g(z)$ for some $k \in \mathbb{Z}$ and further $\displaystyle g(z) = g(0) \frac{h(z)}{k(z)}$, where $\displaystyle  h(z) = \prod_{n=1}^s (1 - \frac{z}{a_n})$ and $\displaystyle k(z)=\prod_{n=1}^t (1 - \frac{z}{b_n})$ for some finite set $\{a_1, a_2, ..., a_s, b_1, b_2, ..., b_t\}$. Construct the operator $D_{h(z)}$ and $D_{k(z)}$ as given in the previous chapter. The following theorem tells us that the determinant exactly recovers the given function $f$.

\begin{thm}\label{rationalthm}
If $f(z)$ is a rational function as above, then 
\begin{equation}
f(z) = z^k g(0) \frac{\det_1 (I - z D_{h(z)})}{\det_1(I-zD_{k(z)})}.
\end{equation} 
\end{thm}
\begin{proof}
Write out $f(z)$ as given in the preceding paragraph. Then consider the finite sequences $Z = \{\frac{1}{a_1}, \frac{1}{a_2}, ..., \frac{1}{a_s}\}$ and $P = \{\frac{1}{b_1}, \frac{1}{b_2}, ..., \frac{1}{b_t}\}$. The operators $D_Z$ and $D_P$ are both trivially trace class since both are created from finite sequences. Hence, we may apply the $1$-regularized determinant (really, just the normal Fredholm determinant since both operators are trace class) we defined to obtain
\begin{equation}
\frac{\det_1 (I - z D_{h(z)})}{\det_1(I-zD_{k(z)})} =  \frac{\prod_{n=1}^s (1 - \frac{z}{a_n})}{\prod_{n=1}^t (1 - \frac{z}{b_n})} =  \frac{g(z)}{g(0)}
\end{equation}
and thus, $f(z) = z^k g(0) \frac{\det_1 (I - z D_{h(z)})}{\det_1(I-zD_{k(z)})} $. \qedsymbol
\end{proof}

We consider $D_{h(z)}$ to be the analog of Frobenius for $h(z)$ (zeros) whereas $D_{k(z)}$ would be the analog for $k(z)$ (poles). Then this ratio of determinants would be considered a graded determinant associated with the Frobenius of the divisor (zeros minus poles) of $g(z)$. 

\subsection{Zeta Functions of Curves Over Finite Fields}

Recall that the zeta function of a (smooth, algebraic) curve $Y$ over the finite field $\mathbb{F}_q$ is defined as $\zeta_Y(s) = \text{exp} \left (\displaystyle \sum_{n=1}^\infty {\frac{Y_n}{n} q^{-ns}} \right )$. The proof of the Weil conjectures expressed this function as an alternating product of determinants as follows: 
\begin{equation*}
\zeta_Y(s) = \frac{\det(I - F^* q^{-s} | H^1)}{\det(I - F^* q^{-s} | H^0)\det(I - F^* q^{-s} | H^2)}.
\end{equation*}
One of the Weil conjectures, that $\zeta_Y( s)$ is a rational function of $q^{-s}$, then followed from this formula. Thus we may apply the result in the previous section about rational functions to obtain the following theorem.

\begin{thm}
Let $Y$ be a smooth, projective, geometrically connected curve over $\mathbb{F}_q$, the field with $q$ elements. Write $\zeta_Y(s) = \frac{f(q^{-s})}{g(q^{-s})}$ with $f(z), g(z)$ both polynomials. Then
\begin{equation}
\zeta_Y(s) = \frac{\det_1 \left (I - q^{-s} D_{f}\right )}{\det_1 \left ( I - q^{-s} D_{g}\right )}.
\end{equation}
\end{thm}
\begin{proof}
We have that $\frac{f(z)}{g(z)}$ is a rational function of $z$. Thus by the rational function result:
$\frac{f(z)}{g(z)} = \frac{\det_1 (I - z D_{f(z)})}{\det_1(I-zD_{g(z)})} $ and so replacing with $z = q^{-s}$ gives: $\zeta_Y(s) = \frac{\det_1 \left (I - q^{-s} D_{f}\right )}{\det_1 \left ( I - q^{-s} D_{g}\right )}$, as desired. \qedsymbol
\end{proof}

Note that the results in this subsection can be extended in a straightforward manner to the zeta function of a (smooth, algebraic) $d$-dimensional variety over $\mathbb{F}_q$, where the integer $d \geq 1$ is arbitrary. 

\subsection{The Gamma Function}

The next meromorphic function that we will turn our attention to is the Gamma function, defined initially by $\displaystyle \Gamma(z) = \int_{0}^\infty {x^{z-1} e^{-x} dx}$ for $\Re(z) >1$. This function has numerous applications in many branches of mathematics, including our focus - number theory. One point of interest is that this function gives a meromorphic continuation to all of $\mathbb{C}$ of the factorial function on integers. It also appears in the functional equation for the Riemann zeta function. We have the following well-known properties of the Gamma function:

\begin{thm} For $z \in \mathbb{C}$, $z \notin \{0, -1, -2, ...\}$, we have\\
1) $\Gamma(z+1) = z \Gamma(z)$.\\ 
2) $\Gamma(n) = (n-1)!$ for $n \in \mathbb{N}$.\\ 
3) $\displaystyle \Gamma(z) = \frac{e^{-\gamma z}}{z} \prod_{n=1}^\infty \left ( \left ( 1+ \frac{z}{n} \right)^{-1} e^{\frac{z}{n}}\right )$.
\end{thm}

This infinite product representation for $\Gamma(z)$ allows us to now show that we can recover the function from the determinant of the operator construction we have laid out.

\begin{thm}\label{gammathm}
We have that for all $z \in \mathbb{C}$,
\begin{equation}\label{gammaeqn}
\Gamma(z) = \frac{e^{-\gamma z}}{z} \frac{1}{\det_2(I - z D_{z\Gamma(z)})}.
\end{equation}
\end{thm}
\begin{proof}
We will apply our construction to the function $g(z) = \frac{1}{z\Gamma(z)}$. This function is entire and has a simple zero at each negative integer. Note that the residue of $\Gamma(z)$ at $z=0$ is $1$; so that $g(0) = 1$. Now if we consider the sequence, $Z = (-\frac{1}{n})$, of reciprocals of zeros of $g(z)$, we see that it is not a summable series but that it is square summable. This means the associated operator $D_Z$ is not trace class, but only Hilbert--Schmidt. This forces us to use $\det_2$ in our definition of the regularized determinant. In fact,
\begin{equation}
\textstyle \det_2(I - z D_Z) = \displaystyle \prod_{n=1}^\infty \left [(1+\frac{z}{n})e^{-\frac{z}{n}} \right ].
\end{equation}
This then leads to the following computation:
\begin{align*}
\textstyle \frac{1}{\det_2 (I - z D_{z\Gamma(z)})} & = \textstyle \det_2(I-zD_Z)^{-1} \\
& =  \left (\prod_{n=1}^\infty \left [\left (1+\frac{z}{n}\right )e^{-\frac{z}{n}} \right ] \right )^{-1}\\
& =\prod_{n=1}^\infty \left [\left (1+\frac{z}{n}\right )^{-1}e^{\frac{z}{n}} \right ] \\
& = z e^{\gamma z} \Gamma(z).
\end{align*}
Therefore, we conclude that $\Gamma(z)$ is given by Equation (\ref{gammaeqn}), as desired. \qedsymbol
\end{proof}

%Before we turn to the Riemann zeta function, we will need to slightly modify the determinant formula above to obtain a formula for $\Gamma\left(\frac{s}{2}\right)$ as that will appear in the completed Riemann zeta function $\xi(s)$.

%\begin{corollary}
%$\Gamma\left(\frac{s}{2}\right) = \frac{2e^{-\gamma s}}{s} \frac{1}{\det_{2} (I - sD_{s\Gamma\left (\frac{s}{2}\right)})}$.
%\end{corollary}
%\begin{proof}
%As before:
%\begin{align*}
%\textstyle \det_{2} (I - s D_{s\Gamma\left(\frac{s}{2}\right)}) & = \det(I-zD_\mathcal{Z}) \textstyle \det_2(I-zD_\mathcal{P})^{-1} \\
%& = \left (\prod_{n=1}^\infty \left [\left (1+\frac{s}{2n}\right )e^{-\frac{s}{2n}} \right ] \right )^{-1}\\
%& = \frac{s}{2} e^{\frac{\gamma s}{2}} \Gamma\left (\frac{s}{2}\right )
%\end{align*}
%This then gives $\Gamma \left ( \frac{s}{2} \right ) = \frac{2e^{-\gamma s}}{s} \det_{1,2} (I - sD_{s\Gamma\left (\frac{s}{2}\right)})$.
%\end{proof}

\subsection{The Riemann Zeta Function}

Next, we turn our attention to another important example, the Riemann zeta function. First, we will consider the well-known Euler product expression for $\zeta(s)$. 

\begin{thm}
For $s \in \mathbb{C}$, with $\Re(s) > 1$, 
\begin{equation*}
\zeta(s) = \prod_{p} (1-p^{-s})^{-1},
\end{equation*}
where the product is taken over all prime numbers $p$.
\end{thm}

To use our formulation, let $\phi(z) = 1-z$. Then, by the result for rational functions in Section 5.1 (Theorem \ref{rationalthm}), we have that $\phi(z) = \det(I - z D_\phi)$, which is true for every value of $z \neq 1$. Letting $z = p^{-s}$ then gives $(1-p^{-s})^{-1} = (\det(I - p^{-s} D_\phi))^{-1}$ for $s \neq \frac{2\pi ik}{\log p}$, $k \in \mathbb{Z}$. This leads to the following operator based Euler product:

\begin{thm}
For $s \in \mathbb{C}$, with $\Re(s) > 1$, 
\begin{equation}
\zeta(s) = \prod_{p} (\det(I - p^{-s} D_\phi))^{-1},
\end{equation}
where the product is taken over the primes $p$.
\end{thm}
\begin{proof}
We simply apply the determinant equality to each term in the infinite product and then use the standard Euler product convergence. Note that for $\Re(s) > 1$, we never have $s = \frac{2\pi ik}{\log p}$ for any integer $k$; so that the determinant equality does apply at each prime $p$. \qedsymbol
\end{proof}

The completed zeta function, $\xi(s) = \frac{1}{2} \pi^{-\frac{s}{2}} s(s-1) \Gamma \left (\frac{s}{2} \right ) \zeta(s)$, is an entire function whose zeros all lie in the critical strip $\{s \in \mathbb{C}: 0 < \Re(s) < 1\}$ and coincide with the critical zeros of $\zeta(s)$. We have the following well-known product representation for $\xi(s)$.

\begin{thm}
For $s \in \mathbb{C}$ with $\Re(s) > 1$,
\begin{equation}
\xi(s) = \frac{1}{2} \pi^{-\frac{s}{2}} e^{\left (\log(2\pi)-1-\frac{\gamma}{2}\right )s} \prod_{\rho}{\left [ \left ( 1 - \frac{s}{\rho}\right ) e^{\frac{s}{\rho}}\right ]},
\end{equation}
where $\gamma$ denotes Euler's constant and the infinite product over $\rho$ is taken over all of the zeros of $\xi(s)$, which are the nontrivial (or critical) zeros of $\zeta(s)$. 
\end{thm}

Now if we wish to express $\xi(s)$ by using the determinant construction in this paper, we need to consider $Z = \{ \frac{1}{\rho}\}$ and the convergence of $\sum_{\rho} \frac{1}{\rho^p}$. It is proven in \cite{Edw} that this series converges for $p=1$, but only conditionally and so we will need $p=2$ to get the absolute convergence needed for $D_Z \in J_2$. Thus we must consider the determinant $\det_{2} (I - s D_{\xi(s)})$. This suggests the following theorem.

\begin{thm}\label{xithm}
For all $s \in \mathbb{C}$, 
\begin{equation}
\xi(s) = \frac{1}{2} \pi^{-\frac{s}{2}} e^{\left (\log(2\pi)-1-\frac{\gamma}{2}\right )s} \textstyle \det_{2} (I - s D_\xi).
\end{equation}
\end{thm}
\begin{proof}
From the preceding discussion, we begin by defining $Z = \{ \frac{1}{\rho}\}$, and constructing $D_{\xi(s)} = D_Z$. Then we calculate:
\begin{align*}
\textstyle \det_{2} (I - s D_{\xi(s)}) & = \textstyle \det_{2} (I - s D_Z) \\
& = \displaystyle \prod_{\rho} \left [\left ( 1 - \frac{s}{\rho} \right ) e^{\frac{s}{\rho}} \right ]  \\
& = \frac{2\xi(s)}{\pi^{-\frac{s}{2}} e^{\left (\log(2\pi)-1-\frac{\gamma}{2} \right )s}}.
\end{align*}
Thus, $ \xi(s) = \frac{1}{2} \pi^{-\frac{s}{2}} e^{\left (\log(2\pi)-1-\frac{\gamma}{2}\right )s} \det_{2} (I - s D_\xi)$, as desired. This result is first obtained for $\Re(s) > 1$, and then upon analytic continuation, for all $s \in \mathbb{C}$. Indeed, both $\xi(s)$ and the regularized determinant define entire functions of $s$. \qedsymbol
\end{proof}

We can then combine the results for $\xi(s)$ and $\Gamma(s)$ to give an expression for $\zeta(s)$ in a similar spirit to the representation of zeta functions of curves over finite fields, as follows.

\begin{thm}\label{zetathm}
For all $s \in \mathbb{C}$, 
\begin{equation}
\zeta(s) = -\frac{e^{(\log(2\pi)-1)s}}{2} \frac{\det_2(I - \frac{s}{2} D_{\frac{s}{2}\Gamma(\frac{s}{2})})) \textstyle \det_{2} (I - s D_\xi)}{\textstyle \det_1(I-sD_\phi)},
\end{equation}
where $\det_2(I - \frac{s}{2} D_{\frac{s}{2}\Gamma(\frac{s}{2})}))$ gives the trivial zeros of $\zeta(s)$, $\textstyle \det_{2} (I - s D_\xi)$ gives the critical zeros of zeta, and $\textstyle \det_1(I-sD_\phi)$ gives the single pole at $s=1$ with $\phi(s) := 1-s$. 
\end{thm}
\begin{proof}
We first recall the following three equations (see, in particular, Theorems \ref{gammathm} and \ref{xithm}), valid for all $s \in \mathbb{C}$:
\begin{equation*}
\Gamma\left ( \frac{s}{2} \right ) = \frac{2e^{-\gamma \frac{s}{2}}}{s} \frac{1}{\textstyle \det_2(I - \frac{s}{2} D_{\frac{s}{2}\Gamma(\frac{s}{2})})},
\end{equation*}
\begin{equation*}
\xi(s) = \frac{1}{2} \pi^{-\frac{s}{2}} s(s-1) \Gamma \left (\frac{s}{2} \right ) \zeta(s),
\end{equation*} 
and
\begin{equation*}
\xi(s) = \frac{1}{2} \pi^{-\frac{s}{2}} e^{\left (\log(2\pi)-1-\frac{\gamma}{2}\right )s} \textstyle \det_{2} (I - s D_\xi).
\end{equation*}
We then solve for $\zeta(s)$ in the middle equation and substitute the other two to obtain successively: 
\begin{align*}
\zeta(s) & = \frac{2\pi^{\frac{s}{2}}\xi(s)}{s(s-1) \Gamma (\frac{s}{2})}\\
& =  2\pi^{\frac{s}{2}} \cdot \frac{\frac{1}{2} \pi^{-\frac{s}{2}} e^{\log(2\pi)-1-\frac{\gamma}{2}s} \det_{2} (I - s D_\xi)}{s(s-1) \frac{2e^{-\gamma s}}{s} ( \det_2(I - \frac{s}{2} D_{\frac{s}{2}\Gamma(\frac{s}{2})}))^{-1}}\\
& = \frac{e^{(\log(2\pi)-1)s}}{2} \frac{\det_2(I - \frac{s}{2} D_{\frac{s}{2}\Gamma(\frac{s}{2})})) \textstyle \det_{2} (I - s D_\xi)}{(s-1)}.
\end{align*}
Finally, letting $\phi(s) = 1-s$ and using Theorem \ref{rationalthm}, since $\phi$ is a rational function, we can replace $s-1 = - \textstyle \det_1(I-sD_\phi)$ and obtain the desired equation. \qedsymbol
\end{proof}

We will conclude this section with a different approach that gives an equivalent criterion for the Riemann Hypothesis. Let $Z$ be the set of zeros of the function defined by $\hat{\xi}(s) = \xi\left ( \frac{1}{2}+is \right )$. Construct the operator $D_Z = D_{\hat{\xi}}$. This leads to the following result.

\begin{thm}\label{RHeq}
The operator $D_{\hat{\xi}}$ is self-adjoint if and only if the Riemann hypothesis is true.
\end{thm}
\begin{proof}
This follows directly from part (2) of Theorem \ref{mainthm} and the fact that the Riemann Hypothesis says that the zeros of $\xi\left ( \frac{1}{2}+is \right )$ must all be real. \qedsymbol
\end{proof}

It should be stressed that Theorem \ref{RHeq} does not as yet provide an approach to the Riemann hypothesis, for some of the reasons outlined in Section 6. We also note that Theorem \ref{xithm} is potentially more useful than Theorem \ref{zetathm} (in part because it does not involve a determinant associated with the gamma function). 

\subsection{Hadamard's Factorization Theorem of Entire Functions}

In this section, we observe that the theory presented here is quite general. It will apply to all entire functions of finite order. We will begin with an overview of the concepts of rank, genus and order of an entire function as described in \cite{Con}. 

\begin{definition}
Let $f$ be an entire function with (nonzero) zeros $\{a_1, a_2, ...\}$, repeated according to multiplicity and arranged such that $|a_1| \leq |a_2| \leq \cdots$. Then $f$ is said to be of \emph{finite rank} if there is a nonnegative integer $p$ such that $\sum_{n=1}^\infty |a_n|^{-(p+1)} < \infty$. If $p$ is the smallest integer such that this occurs, then $f$ is said to be of \emph{rank} $p$; a function with only a finite number of zeros has rank 0. A function is said to be of \emph{infinite rank} if it is not of finite rank.
\end{definition}

In order to define the genus of an entire function, we need to define what it means for an entire function to be written in standard form, which will require the following definition.

\begin{definition}
For $n \in \mathbb{N}_0$, define the \emph{elementary factor} 
\begin{equation*}
E_n(z) = \begin{cases} (1-z), & \text{if }n=0, \\ (1-z)\exp(\frac{z}{1}+\frac{z^2}{2}+\cdots \frac{z^n}{n}), & \text{if }n \geq 1.\end{cases}
\end{equation*}
\end{definition}

To justify the definition of elementary factor, simply note that if $\sum_{n=1}^\infty |a_n|^{-(p+1)} < \infty$, then the infinite product $\prod_{n=1}^\infty E_p(\frac{z}{a_n})$ converges uniformly on compact subsets of $\mathbb{C}$ and defines an entire function with (nonzero) zeros at the complex numbers $a_n$, $n\geq 1$. The exponential factor is what is needed to ensure convergence of the infinite product. With this definition in hand, we can, in turn, define the genus of an entire function:

\begin{definition}
An entire function $f$ has \emph{finite genus} if the following statements hold: 1) $f$ has finite rank $p$ and 2) $f(z) = z^m e^{g(z)}  \prod_{n=1}^\infty E_p \left (\frac{z}{a_n} \right )$, where $g(z)$ is a polynomial of degree $q$. In this case, the \emph{genus} of $f$ is defined by $\mu = \max(p, q)$. 
\end{definition}

We now define the order of an entire function:

\begin{definition}
An entire function $f$ is said to be of \emph{finite order} if there is a nonnegative constant $a$ and and $r_0 > 0$ such that $|f(z)| < \exp(|z|^a)$ for $|z| > r_0$. If $f$ is not of finite order, then $f$ is said to be of \emph{infinite order}. If $f$ is of finite order, then the number $\lambda = \inf \{a: |f(z)| < \exp(|z|^a)\text{ for }|z|\text{ sufficiently large}\}$ is called the \emph{order} of $f$.
\end{definition}

Thus the order of an entire function $f$ is a measure of the growth of $|f(z)|$ as $|z| \to \infty$ whereas the rank of $f$ is based on the growth of the $n^{th}$ smallest zero as $n \to \infty$. From the definitions, there is no inherent relationship between the two concepts, but with the following version of the Hadamard factorization theorem, we see that they are in fact closely related:

\begin{thm}
\emph{(Hadamard's Factorization Theorem)} 
If $f(z)$ is an entire function of finite order $\lambda$, then $f$ has finite genus $\mu \leq \lambda$ and $f$ admits the following factorization: 
\begin{equation}
f(z) = z^m e^{g(z)} \displaystyle \prod_{n=1}^\infty E_p \left (\frac{z}{a_n} \right ),
\end{equation} 
where $g(z)$ is a polynomial of degree $q \leq \lambda$ and $p = [\lambda]$, the integer part of $\lambda$. In particular, $f$ is of rank not exceeding $p$.
\end{thm}

Now when we apply our operator construction to a given entire function of finite order we obtain a Quantized Hadamard Factorization Theorem.

\begin{thm}\label{QHFT}
\emph{(Quantized Hadamard Factorization Theorem)}
If $f(z)$ is an entire function of finite order $\lambda$, then $f$ admits the following factorization:
\begin{equation}
f(z) = z^m e^{g(z)} \textstyle \det_{p+1}(I - z D_{f(z)}),
\end{equation}
 where $g(z)$ is a polynomial of degree $q \leq \lambda$, and $p = [\lambda]$. 
\end{thm} 
\begin{proof}
By the standard Hadamard factorization theorem, we can write
\begin{equation}
f(z) = z^m e^{g(z)} \prod_{n=1}^\infty E_p(\frac{z}{a_n}),
\end{equation}
where $g(z)$ is a polynomial of degree $q \leq \lambda$ and $p = [\lambda]$, with the rank of $f$ not exceeding $p$. That is, if $\{a_1, a_2, ...\}$ is the multiset of zeros of $f(z)$ including multiplicity, then $ \sum_{n=1}^\infty \frac{1}{|a_n|^{p+1}} < \infty$. Thus if $Z = \{ \frac{1}{a_1}, \frac{1}{a_2}, ...\}$, the associated operator $D_Z \in J_{p+1}$. Then we can calculate successively:
\begin{align*}
\textstyle \det_{p+1} (I - z D_{f(z)}) & = \textstyle \det_{p+1} (I - z D_Z) \\
& = \displaystyle \prod_{n=1}^\infty \left [\left ( 1 - \frac{z}{a_n} \right ) \exp \left (\displaystyle \sum_{j=1}^p \frac{z^j}{ja_n^j} \right ) \right ]  \\
& = \displaystyle \prod_{n=1}^\infty E_p \left (\frac{z}{a_n} \right)\\
& = \frac{f(z)}{z^m e^{g(z)}}.
\end{align*}
Thus we have that $f(z) = z^m e^{g(z)} \det_{p+1}(I - z D_{f(z)})$, as desired. \qedsymbol
\end{proof}

In the above proof, we see that the extra convergence factor in the regularized determinants is exactly the same as the one for the elementary factor in the infinite product representation of entire functions, which validates, in some sense, the choice in this paper for the type of regularized determinants as those based on trace ideals. Thus the convergence factors needed in the usual Hadamard Factorization Theorem have an interpretation here relating to $\text{Tr}(D^n)$. 
 
\section{Conclusion}
We ended the previous section by giving what we called the Quantized Hadamard Factorization Theorem. This showed that the construction given in this paper can apply to any entire function of finite order and then, by taking ratios of determinants, can be extended to any meromorphic function which is a ratio of two such entire functions. This was worked out explicitly for the Riemann zeta function (see Theorems \ref{xithm} and \ref{zetathm} above), and it has also been worked out by the authors for zeta functions of self-similar strings, both in the lattice and nonlattice case. (See \cite[Chapters 2 and 3]{LapvFr} for background on self-similar fractal strings.) However, there was nothing in the construction preventing us from applying our results to even more general number-theoretic functions. In particular, a natural idea would be to try to extend our determinant formulas to other L-functions (see \cite{Sar}). Could we then apply this construction to any zeta function (or, at least, to most zeta functions) from arithmetic geometry? This would require, essentially, knowledge about the existence of suitable meromorphic extensions of such functions, as well as about the asymptotic behavior of the zeros and poles of such extensions. Phrased differently, the $L$-functions for which our methods could be applied are those which can be suitably completed to become entire functions of finite order (or ratios of such entire functions). Furthermore, this naturally brings the consideration of the Selberg class of functions. See \cite{Sar} or \cite[Appendix E]{ISRZ} (and the many references therein) for a discussion of these functions. 

Another direction to take is to further justify why using ratios of these determinants is the correct method for handling meromorphic functions. In \cite{ISRZ}, the second author considers the properties of the Riemann zeta function as related to supersymmetric theory in physics and this ratio of determinants can be explained as a (regularized) Berezinian determinant from the theory of super linear algebra. However, this will not be further discussed here, but could be crucial for expanding upon the ideas presented in this work.

In this paper, we obtained a quantized version of the Hadamard factorization theorem, Theorem \ref{QHFT}, but we expect to be able to generalize this result to obtain a quantized Weierstrass product formula; see \cite{CobLap}.

With all of the successes obtained here, we must also admit the failures of this theory, at least in its present stage of its development. The construction of the operator for $\zeta(s)$ explicitly assumed knowledge of the zeros of $\zeta(s)$ and thus one could never hope to prove (RH) directly with this method. However, if we could find a different way to obtain the same function, by comparison you could extrapolate the zeros as was done with the Weil conjectures. That is, we need a suitable geometry and cohomology theory that would result in the same ratios of determinants of these operators. In the Weil conjectures, the geometry or points on the curve (over $\overline{\mathbb{F}_q}$, the algebraic closure of $\mathbb{F}_q$) corresponded to the fixed points of powers (or iterates) of the Frobenius operator. (Recall from our discussion just prior to Section 4.1 that in our context, the ``fractal cohomology space'' would seem to be the total eigenspace $E_Z$ to which we restricted the original generalized Polya--Hilbert operator, viewed as Frobenius acting on an appropriate analog of the underlying ``curve''.) Could we then consider the fixed points of the operator constructed in this paper? Analytically, this can be done by considering a suitable notion of generalized eigenfunctions (viewed as generalized tempered distributions). Thus far, however, this idea has not led to the development of a suitable working theory for the geometry underlying $\zeta(s)$. Providing an appropriate geometric framework is one of our main long-term objectives for future research on this subject. 

Another interesting and related question (connected, in particular, to our discussion in Sections 1.1 and 1.3) is whether the still conjectural fractal cohomology theory satisfies a suitable analogue of the Lefschetz Fixed Point Formula (as stated in Theorem \ref{lefthm}) for the counterpart of Frobenius.

One additional plan that we are currently working on is to rephrase the construction we have described here as a cohomology of sheaves in order to properly transition from the local setup given in this paper to a more global approach that might give new and useful information.

\bibliographystyle{plain}
\bibliography{PaperRef}

\begin{thebibliography}{10}

\bibitem{Art}
E.~Artin.
\newblock Quadratische {K}{\"o}rper im gebiet der h{\"o}heren {K}ongruenzen,
  {I} and {II}.
\newblock {\em Math. Zeitschr.}, \textbf{19}:153--206 and 207--246, 1924.

\bibitem{AtzBri}
A.~Atzmon and B.~Brive.
\newblock Surjectivity and invariant subspaces of differential operators on
  weighted {Bergman} spaces of entire functions.
\newblock In A.~Borichev, H.~Hedenmalm, and K.~Zhu, editors, {\em Contemporary
  {Mathematics}}, volume 404, pages 27--39. American Mathematical Society,
  Providence, R. I., 2006.

\bibitem{Ber}
M.~V. Berry.
\newblock The {B}akerian lecture: {Q}uantum chaology.
\newblock {\em Proceedings of the Royal Society Ser. A}, \textbf{413}:183--198,
  1987.

\bibitem{BerKea}
M.~V. Berry and J.~P. Keating.
\newblock The {R}iemann zeros and eigenvalue asymptotics.
\newblock {\em SIAM Rev.}, \textbf{41}:236--266, 1999.

\bibitem{CobLap}
T.~Cobler and M.~L. Lapidus.
\newblock Zeta functions and {W}eierstrass' factorization theorem via
  regularized determinants and infinitesimal shifts in weighted {B}ergman
  spaces.
\newblock 2017.
\newblock In preparation.

\bibitem{Conn}
A.~Connes.
\newblock Trace formula in noncommutative geometry and the zeros of the
  {R}iemann zeta function.
\newblock {\em Selecta Mathematica}, New Series \textbf{5}:29--106, 1999.

\bibitem{Con}
J.~B. Conway.
\newblock {\em Functions of One Complex Variable}, volume~11 of {\em Graduate
  Texts in Mathematics}.
\newblock Springer-Verlag, New York, second edition, 1995.

\bibitem{Del1}
P.~Deligne.
\newblock La conjecture de {W}eil: {I}.
\newblock {\em Publications Math{\'e}matiques de l'IH{\'E}S},
  \textbf{43}:273--307, 1974.

\bibitem{Del2}
P.~Deligne.
\newblock La conjecture de {W}eil: {II}.
\newblock {\em Publications Math{\'e}matiques de l'IH{\'E}S},
  \textbf{52}:137--252, 1980.

\bibitem{Den2}
C.~Deninger.
\newblock Some analogies between number theory and dynamical systems on
  foliated spaces.
\newblock In {\em Proc. Internat. Congress Math. (Berlin, 1998)}, volume~I,
  pages 163--186.
\newblock Documenta Math. J. DMV (Extra Volume ICM 98).

\bibitem{Den1}
C.~Deninger.
\newblock Evidence for a cohomological approach to analytic number theory.
\newblock In A.~Joseph, F.~Mignot, F.~Murat, B.~Prum, and R.~Rentschler,
  editors, {\em First {European} {Congress} of {Mathematics}}, volume~3 of {\em
  Progress in {Mathematics}}, pages 491--510. Birkh\"{a}user Basel, 1994.

\bibitem{Die}
J.~Dieudonn{\'e}.
\newblock On the history of the {W}eil conjectures.
\newblock {\em The Mathematical Intelligencer}, \textbf{10}:7--21, 1975.

\bibitem{Edw}
H.~M. Edwards.
\newblock {\em Riemann's {Zeta} {Function}}.
\newblock Dover Publications, Mineola, NY, {Dover} edition, 2001.

\bibitem{Gro1}
A.~Grothendieck.
\newblock The cohomology theory of abstract algebraic varieties.
\newblock In {\em Proc. Internat. Congress Math. (Edinburgh, 1958)}, pages
  103--118. Cambridge Univ. Press, New York, 1960.

\bibitem{Gro2}
A.~Grothendieck.
\newblock Formule de {L}efschetz et rationalit{\'e} des fonctions l.
\newblock {\em S{\'e}minaire N. Bourbaki}, Exp. 279:41--55, 1964--1966.

\bibitem{Gro3}
A.~Grothendieck.
\newblock Standard conjectures on algebraic cycles.
\newblock {\em Algebraic {Geometry} ({Internat.} {C}olloq., {T}ata {I}nst.
  {F}und. {R}es., {B}ombay, 1968)}, pages 193--199, 1969.

\bibitem{Har2}
S.~Haran.
\newblock {\em {The Mysteries of the Real Prime}}, volume~25 of {\em London
  Math. Soc. Monographs, New Ser.}
\newblock Oxford Science Publications, Oxford Univ. Press, Oxford, 2001.

\bibitem{Has}
H.~Hasse.
\newblock Abstrakte {B}egr{\"u}ndung der komplexen {M}ultiplikation und
  {R}iemannsche {V}ermutung in {F}unktionenk{\"o}rpern.
\newblock {\em Anh. Math. Sem. Hamburg}, \textbf{10}:325--348, 1934.

\bibitem{Hed}
H.~Hedenmalm, B.~Korenblum, and K.~Zhu.
\newblock {\em Theory of Bergman Spaces}.
\newblock Graduate Texts in Mathematics, volume 199. Springer, New York, 2000.

\bibitem{HerLap2}
H.~Herichi and M.~L. Lapidus.
\newblock Riemann zeros and phase transitions via the spectral operator on
  fractal strings.
\newblock {\em J. Phys. A: Math. Theor.}, \textbf{45}, 374005, 23pp., 2012.

\bibitem{HerLap3}
H.~Herichi and M.~L. Lapidus.
\newblock {\em Fractal complex dimensions, {Riemann} hypothesis and
  invertibility of the spectral operator}, volume 600 of {\em Contemporary
  Mathematics}, pages 51--89.
\newblock Amer. Math. Soc., Providence, R. I., 2013.

\bibitem{HerLap4}
H.~Herichi and M.~L. Lapidus.
\newblock Truncated infinitesimal shifts, spectral operators and quantized
  universality of the {R}iemann zeta function.
\newblock {\em Annales de la Facult\'e des Sciences de Toulouse}, No. 3,
  \textbf{23}:621--664, 2014.
\newblock [Special issue dedicated to Christophe Soul\'e on the occasion of his
  60th birthday.].

\bibitem{HerLap1}
H.~Herichi and M.~L. Lapidus.
\newblock {\em Quantized Number Theory, Fractal Strings and the Riemann
  Hypothesis: From Spectral Operators to Phase Transitions and Universality}.
\newblock Research Monograph. World Scientific, Singapore and London, to
  appear, 2017.
\newblock approx. 330 pages.

\bibitem{Kat}
N.~Katz.
\newblock An overview of {D}eligne's proof of the {R}iemann hypothesis for
  varieties over finite fields.
\newblock In {\em Proc. Symposia Pure Math.}, volume~28, pages 275--305. Amer.
  Math. Soc., Providence, R. I., 1976.

\bibitem{ISRZ}
M.~L. Lapidus.
\newblock {\em In Search of the {Riemann} Zeros: Strings, Fractal Membranes and
  Noncommutative Spacetimes}.
\newblock American Mathematical Society, Providence, R. I., 2008.

\bibitem{Lap}
M.~L. Lapidus.
\newblock Towards quantized number theory: spectral operators and an asymmetric
  criterion for the {Riemann} hypothesis.
\newblock {\em Philosophical Transactions of the Royal Society Ser. A},
  \textbf{373}, No. 2047, 24pp., 2015.

\bibitem{LapMai}
M.~L. Lapidus and H.~Maier.
\newblock The {Riemann} hypothesis and inverse spectral problems for fractal
  strings.
\newblock {\em Journal of the London Mathematical Society}, No. 1,
  \textbf{52}(2):15--34, 1995.

\bibitem{LapvFr}
M.~L. Lapidus and M.~van Frankenhuijsen.
\newblock {\em Fractal Geometry, Complex Dimensions and Zeta Functions:
  Geometry and Spectra of Fractal Strings}.
\newblock Springer Monographs in Mathematics. Springer, New York, 2013.
\newblock Second revised and enlarged edition of the 2006 edition.

\bibitem{Oort}
F.~Oort.
\newblock {The Weil} conjectures.
\newblock {\em Nw. Archief v. Wiskunde 5 Ser. \textbf{15}, No. 3}, pages
  211--219, 2014.

\bibitem{Pat}
S.~J. Patterson.
\newblock {\em An {I}ntroduction to the {T}heory of the {R}iemann
  {Z}eta-{F}unction}.
\newblock Cambridge Univ. Press, Cambridge, 1988.

\bibitem{Rie}
B.~Riemann.
\newblock Ueber die {A}nzahl der {P}rimzahlen unter einer gegebenen
  {G}r{\"o}sse.
\newblock {\em Monatsb. der Berliner Akad.}, pages 671--680, 1858/1860.

\bibitem{Rud}
W.~Rudin.
\newblock {\em Functional Analysis}.
\newblock Mc{G}raw {Hill}, New York, second edition, 1991.

\bibitem{Sar}
P.~Sarnak.
\newblock L-functions.
\newblock In {\em Proc. Internat. Congress Math.}, volume~I, pages 453--465.
  Berlin, 1998.
\newblock Documenta Math. J. DMV (Extra Volume ICM 98).

\bibitem{Sch}
F.~K. Schmidt.
\newblock Analytische {Z}ahlentheorie in {K}{\"o}rpern der {C}harakteristik
  $p$.
\newblock {\em Math. Zeitschr.}, \textbf{33}:1--32, 1931.

\bibitem{Sim2}
B.~Simon.
\newblock Notes on infinite determinants of {Hilbert} space operators.
\newblock {\em Advances in Math.}, \textbf{24}:244--273, 1977.

\bibitem{Sim}
B.~Simon.
\newblock {\em Trace Ideals and Their Applications}, volume 120 of {\em
  Mathematical Surveys and Monographs}.
\newblock American Mathematical Society, Providence, R. I., second edition,
  2005.

\bibitem{Tit}
E.~C. Titchmarsh.
\newblock {\em The {T}heory of the {R}iemann {Z}eta {F}unction}.
\newblock Oxford Univ. Press, Oxford Mathematical Monographs, Oxford, 1986.
\newblock Second edition (revised by {D}. {R}. {H}eath-{B}rown).

\bibitem{Wei2}
A.~Weil.
\newblock Sur les fonctions alg{\'e}briques {\`a} corps de constantes fini.
\newblock {\em C. R. Acad. Sci. Paris}, \textbf{210}:592--594, 1940.

\bibitem{Wei3}
A.~Weil.
\newblock On the {R}iemann hypothesis in function fields.
\newblock {\em Proc. Nat. Acad. Sci., U.S.A.}, \textbf{27}:345--347, 1941.

\bibitem{Wei1}
A.~Weil.
\newblock Numbers of solutions of equations in finite fields.
\newblock {\em Bull. Amer. Math. Soc.}, \textbf{55}:497--508, 1949.

\bibitem{Wei4}
A.~Weil.
\newblock Abstract versus classical algebraic geometry.
\newblock In {\em Proc. Internat. Congress Math. (Amsterdam, 1954)}, volume
  {III}, pages 550--558. Noordhoff, Groningen, and North-Holland, Amsterdam,
  1956.

\bibitem{Wei8}
A.~Weil.
\newblock Sur les formules explicites de la th{\'e}orie des nombres.
\newblock {\em Izv. Mat. Nauk (Ser. Mat.)}, \textbf{36}:3--18, 1972.

\end{thebibliography}

\end{document}